\documentclass[12pt]{amsart}
\usepackage{amsmath,amssymb,amscd,array, mathrsfs }

\usepackage{amsmath,amscd,amssymb,amsthm,array}
\usepackage{color}

\usepackage{amsmath,amssymb,mathrsfs,amsthm, tikz-cd,mathrsfs}
\usepackage{comment}
\def\url#1{\expandafter\s

\tring\csname #1\endcsname}

\def\mmat #1,#2,#3,#4,{\text{\small\arraycolsep=3pt $
\begin{pmatrix}#1&#2\\#3&#4\end{pmatrix}$}}

\usepackage{hyperref}
\usepackage{capt-of}

\usepackage{multirow}
\usepackage[all]{xy}

\usepackage{comments}
\newComments\SBe{Said}{blue}
\newComments\SBo{Sofiane}{blue}
\newComments\AM{Nacer}{blue}
\newComments\DL{DL}{red}
\newComments\QEh{QEh}{blue}

\def\mmat #1,#2,#3,#4,{\text{\small\arraycolsep=3pt $
\begin{pmatrix}#1&#2\\#3&#4\end{pmatrix}$}}

\usepackage{lscape}
\usepackage{tikz-cd}
\usepackage{enumerate}

\usepackage{DLdef1}
\usepackage{multicol}

\def\mmat #1,#2,#3,#4,{\text{\small\arraycolsep=3pt $
\begin{pmatrix}#1&#2\\#3&#4\end{pmatrix}$}}

\renewcommand {\ssbegin}[2][*]
 {\refstepcounter{subsection}%
\if#1*
\addcontentsline{toc}{subsection}{\thesubsection.\hskip 1pc #2}%
\else
\addcontentsline{toc}{subsection}{\thesubsection.\hskip 1pc #2. #1}%
\fi
 \def \secno {\gdef \secno {}{\ssecfont
\thesubsection.\hskip 2ex}%
 }%
 \begin{#2}}

\renewcommand {\sssbegin}[2][*]
 {\refstepcounter{subsubsection}
\if#1*
\addcontentsline{toc}{subsubsection}{\thesubsubsection.\hskip 1pc #2}%
\else
\addcontentsline{toc}{subsubsection}{\thesubsubsection.\hskip 1pc #2. #1}
\fi
 \def \secno {\gdef \secno {}{\ssecfont \thesubsubsection.\hskip 2ex}%
 }%
 \begin{#2}}

\renewcommand {\parbegin}[2][*]
 {\refstepcounter{paragraph}
\if#1*
\addcontentsline{toc}{paragraph}{\theparagraph.\hskip 1pc #2}%
\else
\addcontentsline{toc}{paragraph}{\theparagraph.\hskip 1pc #2. #1}
\fi
 \def \secno {\gdef \secno {}{\ssecfont \theparagraph.\hskip 2ex}%
 }%
 \begin{#2}}

\rmnameii{etr}{etr}
\rmnameii{evv}{ev}
\DeclareMathOperator{\h}{\mathcal{H}}

\DeclareMathOperator{\K}{\mathbb{K}}
\newcommand{\M}{{\cal M}}
\newcommand{\G}{\mathfrak{g}}
\newcommand{\g}{\mathfrak{g}}
\newcommand {\A}{{\cal{A}}}

\newcommand{\Ll}{{\mathrm{L}}}
\newcommand{\Rr}{{\mathrm{R}}}

\newcommand{\al}{\alpha}

\newcommand{\la}{\lambda}

\newcommand{\de}{\delta}  
\newcommand{\prs}{\langle\;,\;\rangle}
\def\br{[\;,\;]}
\newcommand{\esp}{\quad\mbox{and}\quad}

\begin{document}

\title[On flat pseudo-Euclidean solvable Malcev algebras ]{On flat pseudo-Euclidean solvable Malcev algebras  }

\author{Mohamed Boucetta}
\address{Universit\'e Cadi-Ayyad,
	Facult\'e des sciences et techniques,
	B.P. 549, Marrakech, Maroc.}
\email{m.boucetta@uca.ac.ma}

\author{Hamza El Ouali$^{*}$}
\address{Universit\'e Cadi-Ayyad,
	Facult\'e des sciences et techniques,
	B.P. 549, Marrakech, Maroc.}
\email{eloualihamza11@gmail.com}

\renewcommand{\thefootnote}{\fnsymbol{footnote}}
\footnotetext[1]{Corresponding author.}

\author{Oumaima Tibssirte}
\address{ Department of Health and Agro-Industry Engineering, High School of Engineering and Innovation ofMarrakesh, Private University of Marrakesh, Road Amezmiz, Marrakech, Morocco }
\email{oumayma1tibssirte@gmail.com}




\begin{abstract}
A pseudo-Euclidean Malcev algebra is a Malcev algebra equipped with a non-degenerate symmetric bilinear form. In this paper, we introduce a curvature operator for pseudo-Euclidean Malcev algebras, generalizing the notion of curvature for pseudo-Euclidean Lie algebras. We define flat pseudo-Euclidean Malcev algebras and show that every flat Euclidean solvable Malcev algebra which is also a Lie algebra remains flat in the classical Lie algebra curvature sense. Furthermore, we develop the flat double extension construction for flat pseudo-Euclidean Malcev algebras and prove that every flat  Lorentzian Malcev algebra with a degenerate center can be obtained via the flat double extension of a flat Euclidean Malcev algebra. Moreover, we demonstrate that all flat Lorentzian nilpotent Malcev algebras arise from the flat double extension of Euclidean abelian Lie algebras. Finally, we establish that any flat Lorentzian nilpotent Malcev algebra is necessarily a Lie algebra and is flat in the classical Lie algebra curvature sense.
	\end{abstract}


\maketitle

\thispagestyle{empty}
\setcounter{tocdepth}{2}
\section{Introduction}

A \emph{pseudo-Riemannian Lie group} is a Lie group $G$ endowed with a left-invariant pseudo-Riemannian metric $\mu$. Its Lie algebra $\mathfrak{g}=T_eG$, equipped with the scalar product $\prs=\mu_e$, is called a \emph{pseudo-Riemannian Lie algebra} or, equivalently, a \emph{pseudo-Euclidean Lie algebra}. Whenever there exists a torsion-free connection compatible with the metric, it is necessarily unique and is called the \emph{Levi-Civita connection}. This connection induces a bilinear product on $\mathfrak{g}$, called the \emph{Levi-Civita product}, defined by  Koszul’s formula:
$$
2\langle u\bullet v, w\rangle
= \langle [u,v], w\rangle
- \langle [v,w], u\rangle
+ \langle [w,u], v\rangle,
  \qquad \forall u,v,w\in\mathfrak{g}.
  $$
  This product satisfies
  $$
  u\bullet v - v\bullet u = [u,v],\qquad
  \langle u\bullet v, w\rangle = -\langle v, u\bullet w\rangle,
  \qquad \forall u,v,w\in\mathfrak{g}.
  $$
In \cite{Boucetta0}, an analogue of the Levi-Civita product was introduced for pseudo-Euclidean non-associative algebras, in particular anti-commutative algebras, and this construction was subsequently extended to non-associative superalgebras in \cite{BBE}.

The curvature of $(\mathfrak{g},\prs)$ at the identity is
$$
K_0(u,v):=\Ll_{[u,v]}-[\Ll_u,\Ll_v],\qquad \forall u,v\in\mathfrak{g},
$$
where $\Ll_u$ denotes left multiplication, $\Ll_u(v)=u\bullet v$.

A pseudo-Riemannian Lie group $(G,\mu)$ is called \emph{flat} if its curvature vanishes identically; in this case, $(\mathfrak{g},\prs)$ is called a \emph{flat pseudo-Riemannian Lie algebra}. The vanishing of the curvature is equivalent to the fact  that $\mathfrak{g}$,  endowed with the Levi-Civita product, is a left-symmetric algebra, i.e., for any $u,v,w\in \g$
$$(u, v, w) = (v,u,w),$$
where $(u, v, w)=(u\bullet v)\bullet w-u\bullet(v\bullet w)$. So a flat pseudo-Riemannian Lie algebra can be viewed as a left-symmetric algebra with a non-degenerate symmetric bilinear form for which the left
multiplications are skew-symmetric. The study of flat pseudo-Riemannian Lie algebras originated in the Riemannian case. A foundational result, proved by J. Milnor in 1976 \cite{Milnor}, characterizes flat Riemannian Lie algebras: a Lie algebra $\g$ is flat Riemannian if and only if it admits an orthogonal decomposition
$
\g = \mathfrak{b} \oplus \mathfrak{u},
$
where $\mathfrak u$ is an abelian ideal, $\mathfrak b$ is an abelian subalgebra, and for every $b\in \mathfrak b$, the adjoint map $\ad_b$ is skew-symmetric with respect to the inner product, where $\ad_b(a)=[a,b]$, for all $a\in \g$.

In recent years, many authors have investigated flat pseudo-Riemannian Lie algebras of various signatures, using in particular the method of double extension introduced in \cite{Medina}. In \cite{Medina}, it was shown that all flat Lorentzian nilpotent Lie algebras are obtained by the double extension process from Riemannian abelian Lie algebras. Further developments include the study of flat Lorentzian Lie algebras in \cite{Boucetta}, where it was shown that all flat Lorentzian nilpotent Lie algebras with degenerate center are obtained by double extensions from flat Riemannian Lie algebras, and that all flat Lorentzian Lie algebras with a degenerate center are unimodular. Moreover, \cite{Boucetta2} showed that all non-unimodular flat Lorentzian Lie algebras are obtained by double extensions from flat Riemannian Lie algebras. In \cite{Bajo}, a complete classification of flat Lorentzian nilpotent Lie algebras was given. Furthermore, in \cite{Boucetta3}, flat pseudo-Riemannian nilpotent Lie algebras of signature $(2,n-2)$ were studied, showing that all  flat pseudo-Riemannian nilpotent Lie algebras of signature $(2,n-2)$ are obtained by double extensions from flat Lorentzian nilpotent Lie algebras. Finally, in \cite{Boucetta1, Hicham, Hicham1}, particular cases of flat pseudo-Riemannian Lie algebras were studied, highlighting their structural properties.

The notion of \emph{Malcev algebras} was introduced by A. Malcev in the 1950s as a generalization of Lie algebras \cite{Malcev}. He called these objects Moufang-Lie algebras because of their close connection with analytic Moufang loops. A Malcev algebra $(\mathcal M,\br)$ over a field $\mathbb K$ is an anticommutative algebra satisfying the Malcev identity
$$
J(u,v,[u,w]) = [J(u,v,w),u], \qquad \forall u,v,w\in\mathcal M,
$$
where $J(u,v,w)= [[u,v],w]+[[v,w],u]+[[w,u],v]$, denotes the Jacobian operator. When $J(u,v,w)=0$  for all $u,v,w\in\mathcal M$,  the Malcev algebra reduces to a Lie algebra. Hence, Lie algebras form a subclass of Malcev algebras. Malcev algebras play an important role in mathematical physics and in the geometry of smooth loops. Indeed, just as the tangent algebra of a Lie group is a Lie algebra, the tangent algebra of a locally analytic Moufang loop is a Malcev algebra, (see, for instance \cite{Malcev, K, K1, G}).

The notion of a \emph{pre-Malcev algebra}, first introduced by S. Madariaga in \cite{Sara}, consists of a vector space $\mathcal M$ endowed with a bilinear product $\bullet$ satisfying the identity
\begin{align*}
\;\;\;\;\;\; (v& \bullet w)\bullet (u \bullet z) - (w \bullet v) \bullet (u \bullet z) + ((u \bullet v) \bullet w) \bullet z - ((v \bullet u) \bullet w) \bullet z \\
      & - (w \bullet (u \bullet v)) \bullet z + (w \bullet (v \bullet u)) \bullet z + v \bullet ((u \bullet w) \bullet z) - v \bullet ((w \bullet u) \bullet z) \\
   & - u\bullet (v\bullet (w\bullet z)) + w\bullet (u\bullet (v \bullet z)) = 0,
\end{align*}
for all $u,v,w,z \in \mathcal M$. Given a pre-Malcev algebra $(\mathcal M,\bullet)$, the skew-symmetric bracket defined by $
[u,v] = u\bullet v - v\bullet u
$ endows $\mathcal M$ with a Malcev algebra structure (see \cite{Sara}).

A \emph{pseudo-Euclidean Malcev algebra} is a triple
\((\mathcal M,\br,\prs)\), where
\((\mathcal M,\br)\) is a Malcev algebra endowed with a symmetric
non-degenerate bilinear form \(\prs\). In \cite{Boucetta0}, the authors introduced the notion of a
\emph{Levi-Civita product} associated with pseudo-Euclidean non-associative algebras, in particular anti-commutative algebras. On the other hand, in \cite{Sara}, the notion of
a \emph{pre-Malcev algebra} was introduced as a generalization of
pre-Lie algebras, in the sense that every pre-Lie algebra is a
pre-Malcev algebra, while the converse does not hold. In this paper, we introduce a natural \emph{curvature operator} associated with a pseudo-Euclidean Malcev algebra, which generalizes the
curvature construction arising in pseudo-Riemannian Lie algebras.
More precisely, let
\((\mathcal M,\br,\prs)\) be a
pseudo-Euclidean Malcev algebra, and let \(\bullet\) denote the
Levi-Civita product associated with
\((\mathcal M,\br,\prs)\).
For any \(u,v,w\in\mathcal M\), the curvature operator is defined by
\[
K(u,v,w)
= \Ll_{[[u,v],w]}
- \Ll_{[u,w]}\circ \Ll_v
- \Ll_u\circ \Ll_{[v,w]}
+ \Ll_v\circ \Ll_u\circ \Ll_w
- \Ll_w\circ \Ll_v\circ \Ll_u,
\]
where $\Ll_u$ denotes left multiplication \(\Ll_u(v)=u\bullet v\).

The Malcev algebra
\((\mathcal M,\br,\prs)\) is called a
\emph{flat pseudo-Euclidean Malcev algebra} if \(K\equiv 0\).
Consequently, a pseudo-Euclidean Malcev algebra is flat if and only if it
admits a Levi-Civita product endowing it with a pre-Malcev algebra structure. In this
way, a flat pseudo-Euclidean Malcev algebra can be viewed as a
pre-Malcev algebra equipped with a symmetric non-degenerate bilinear form
for which all left multiplications are skew-symmetric.

In this paper, we study flat pseudo-Euclidean Malcev algebras, with a
special emphasis on the Euclidean and Lorentzian solvable Malcev algebra cases. We show
that every  flat Euclidean solvable Malcev algebra is necessarily  a Lie algebra and is flat in
the classical Lie algebra curvature sense. (see Theorem~\ref{EMalcev} and
Corollary~\ref{co1}). Moreover, we introduce the notion of \emph{flat
double extension} for flat pseudo-Euclidean Malcev algebras
(see Theorem~\ref{double-ex1}). This construction is inspired by the
classical double extension introduced by A.~Medina and P.~Revoy
for quadratic Lie algebras \cite{Medina1}. Furthermore, we prove that every flat Lorentzian Malcev algebra whose
center is degenerate can be obtained as a flat double extension of a flat
Euclidean  Malcev algebra (see Theorem \ref{Thpr}). In particular, we show that every flat
Lorentzian solvable Malcev algebra with a degenerate center arises as a flat
double extension of a flat Euclidean Lie algebra (see Corollary \ref{SL}). We also prove that every
flat Lorentzian nilpotent Malcev algebra is obtained as a flat double
extension of a Euclidean abelian Lie algebra, and consequently, any flat
Lorentzian nilpotent Malcev algebra  is necessarily a Lie algebra and is flat in
the classical Lie algebra curvature sense. (see Theorem~\ref{Thnilpotent} and
Proposition~\ref{PLie}). 

\medskip
\noindent\textbf{Structure of the paper}.
The paper is organized as follows. In Section~\ref{s2}, we introduce the curvature operator associated with the
Levi-Civita product of a pseudo-Euclidean Malcev algebra. We study
pseudo-Euclidean Malcev algebras and prove that a quadratic Malcev algebra
is flat if and only if it is nilpotent of nilpotency class at most three.
Recall that a quadratic Malcev algebra is a Malcev algebra endowed with a
symmetric non-degenerate invariant (also called associative)  bilinear form. Quadratic Malcev
superalgebras were previously studied in \cite{ABen, Ba1, Ba2}. Moreover, we
show that if a pseudo-Euclidean Lie algebra is flat in the classical Lie algebra curvature sense, then
it is also flat as a Malcev algebra, while the converse does not hold in
general. We also prove that a Euclidean solvable Malcev algebra is flat if and only
if it is a Lie algebra which is flat in the classical Lie algebra curvature sense.

In Section~\ref{s3}, we introduce the notion of flat double extension for flat
pseudo-Euclidean Malcev algebras and establish the main structural
results associated with this construction.

In Section~\ref{s4}, we study flat Lorentzian Malcev algebras with degenerate
center. We prove that such algebras can be obtained by flat double
extensions of a flat Euclidean Malcev algebras. In particular, we show that every flat
Lorentzian solvable Malcev algebra with a degenerate center arises as a flat
double extension of a  Euclidean Lie algebra which is flat in the classical Lie algebra curvature sense and show
that every flat Lorentzian nilpotent Malcev algebra is necessarily a 
Lorentzian nilpotent Lie algebra which is flat in the classical Lie algebra curvature sense.

Throughout this work, all vector spaces are assumed to be finite dimensional over a field $\mathbb{K}$ of characteristic zero.

A \emph{pseudo-Euclidean vector space} is a finite-dimensional $\mathbb{K}$-vector space $V$ endowed with a non-degenerate symmetric bilinear form
$\prs:V\times V\longrightarrow\mathbb{K}.$ When $\mathbb{K}=\mathbb{R}$, the form $\prs$ has a well-defined signature $(q,n-q)=(-\cdots-,+\cdots+)$, where $n=\dim V$. In particular, $(V,\prs)$ is called \emph{Euclidean} if its signature is $(0,n)$, and \emph{Lorentzian} if its signature is $(1,n-1)$. 
 Throughout the paper, whenever we refer to the signature of a pseudo-Euclidean vector space, we implicitly assume that the base field is $\mathbb{R}$.

\section{Flat Euclidean solvable Malcev algebras}\label{s2}
In this section,  we introduce the
curvature operator associated with the Levi-Civita product of a pseudo-Euclidean Malcev
algebra. We study pseudo-Euclidean Malcev algebras and prove that a quadratic Malcev
algebra is flat  if and only if it is nilpotent of nilpotency class at most three. We show that if a pseudo-Euclidean Lie algebra is flat which is flat in the classical Lie algebra curvature sense, then it is also flat pseudo-Euclidean Malcev algebra, while the converse does not hold in general.
We also prove that a Euclidean solvable Malcev algebra is flat if and only if it is a Lie algebra
which is flat in the classical Lie algebra sense.

\ssbegin{Definition}(\cite{Malcev})
A Malcev algebra is a non-associative algebra \( \M \) with an anti-symmetric
bilinear multiplication \(\br : \M \times \M \to \M\) that satisfies the Malcev identity for \( u, v, w \in \M \),
\begin{equation} \label{eq:Malcev}
J(u, v, [u, w]) = [J(u, v, w), u],
\end{equation}
where the Jacobian \( J(u, v, w) \) is defined as
\(
J(u, v, w) = [[u, v], w] + [[w, u], v] + [[v, w], u].
\)

Expanding the Jacobian, since the characteristic of \( \K \) is not $2$, the Malcev identity \eqref{eq:Malcev} is
equivalent to Sagle’s identity \cite{Sagle}, for \( u, v, w, z \in \M \) we have
\begin{equation} \label{eq:Sagle}
[[u, w], [v, z]] = [[[u, v], w], z] + [[[v, w], z], u] + [[[w, z], u], v] + [[[z, u], v], w].
\end{equation}
\end{Definition}
\ssbegin{Example}
\begin{enumerate}
    \item Lie algebras are examples of Malcev algebras. In fact, every Lie algebra satisfies the Malcev identity, which is a generalization of the Jacobi identity. This makes every Lie algebra a Malcev algebra.
\item Alternative algebras are  Malcev-admissible algebras. Specifically, if \((\M, \cdot)\) is an alternative algebra, then the commutator defined by
\(
[u, v] = u \cdot v- v \cdot u
\)
defines a Malcev structure on \( \M \). These Malcev algebras are referred to as \textit{special} Malcev algebras.
\end{enumerate}
\end{Example}

Since every Lie algebra is a Malcev algebra, it follows that  all Lie-admissible algebras
are also Malcev-admissible. Below, we provide an example of a 4-dimensional Malcev algebra that is not a Lie algebra.
\ssbegin{Example}[\cite{Sagle}]
    Let \( \M \) be a 4-dimensional vector space over a field \( \mathbb{K} \) with basis \( \{e_1, e_2, e_3, e_4\} \). The bilinear product on \( \M \) is defined by the following non-zero products:
\[
[e_1 , e_2] = -e_2, \;
[e_1 , e_3 ]= -e_3, \;
[e_1 , e_4] = e_4, \;
[e_2 , e_3] = 2e_4.
\]
 endows $\M$ with a non-Lie Malcev algebra structure
\end{Example}
\ssbegin{Definition}
    A representation of a Malcev algebra \( (\M, \br) \) on a vector space \( V \) is a linear map $
\rho: \M \to \mathrm{End}(V)
$   satisfying for all \( u, v, w \in \M \) the condition:  
\[
\rho([[u, v], w]) = \rho(u)\rho(v)\rho(w) - \rho(w)\rho(u)\rho(v) + \rho(v)\rho([w, u]) - \rho([v, w])\rho(u).
\]  

\end{Definition}
\ssbegin{Definition}[\cite{Sara}]
A pre-Malcev algebra is a vector space $\M$ endowed with a bilinear product satisfying for any $u,v,w,z \in \M$ the following identity:
\begin{align*}
 (v& \bullet w)\bullet (u \bullet z) - (w \bullet v) \bullet (u \bullet z) + ((u \bullet v) \bullet w) \bullet z - ((v \bullet u) \bullet w) \bullet z \\
& - (w \bullet (u \bullet v)) \bullet z + (w \bullet (v \bullet u)) \bullet z + v \bullet ((u \bullet w) \bullet z) - v \bullet ((w \bullet u) \bullet z) \\
& - u\bullet (v\bullet (w\bullet z)) + w\bullet (u\bullet (v \bullet z)) = 0,
\end{align*}

\end{Definition}
 Let $(\A, \star)$ be a non-associative algebra. On the underlying vector space $\A$, we
define the corresponding commutator 
$$[u, v]_\star :=u\star v - v\star u,\;  \text{ for all } u,v\in\A.$$ 
The algebra $(\A, \br_\star)$ will be denoted by  $\A^-$. The algebra $(\A, \star)$ is called a Malcev-admissible
algebra, if $\A^-$ is a Malcev algebra. 

\ssbegin{Proposition}[\cite{Sara}]\label{Sara}
 Let $(\M, \bullet)$ be a pre-Malcev algebra.  Then $(\M, \bullet)$ is a Malcev-admissible
algebra, i.e., $\M^-$ is a Malcev algebra.
\end{Proposition}

\ssbegin{Definition}
Let $(\A, \star)$ be a non-associative algebra equipped with a non-degenerate symmetric bilinear form 
$\prs$. Such a form is called a pseudo-Euclidean structure on  $\A$ and the triple $(\A, \star,
\prs)$ defines what is termed a 
pseudo-Euclidean non-associative algebra.
\end{Definition}

In \cite{Boucetta0}, the authors introduced a construction analogous to the Levi-Civita product in pseudo-Euclidean Lie algebras, extending it to pseudo-Euclidean non-associative algebras. This provides a unique and natural product in this setting, particularly for anti-commutative algebras.
\ssbegin{Proposition} $($\cite{Boucetta0}$)$ \label{LV}
Let $(\M, \br, \prs)$  be a pseudo-Euclidean anti-commutative nonassociative algebra.
Then there exists a unique product $\bullet$ on the underlying vector space of the algebra $\M$ such that:
\begin{align}\label{compatible}
\langle u \bullet v, w\rangle &=-\langle v , u\bullet w\rangle,\; \\ \label{sanstorsion}
[u,v] &= u\bullet v-v\bullet u, \end{align}
for any $u,v,w\in\M$. More precisely, the product $\bullet$ is defined by
\begin{equation}
\left\langle u\bullet v, w\right\rangle=\frac{1}{2}(\langle[u,v], w\rangle-\langle [v,w], u\rangle+\langle [w,u], v\rangle), 
\label{lv}\end{equation}
for all $u,v,w \in \M$.
The product $\bullet$ is called the Levi-Civita product associated to the pseudo-Euclidean anti-commutative non-associative algebra $(\M,\br,\prs)$.

	\end{Proposition}

Let $(\M, \br, \prs)$ be a pseudo-Euclidean anti-commutative nonassociative algebra, and  $"\bullet"$  will denote the Levi-Civita product. For any $u \in \M$, we define $\Ll_u: \M \to \M$ and $\mathrm{R}_u: \M \to \M$ as the left and right multiplications by $u$ given by
$$ \mathrm{L}_u v = u \bullet v \quad \text{and} \quad \mathrm{R}_u v = v \bullet u. $$
According to the relations \eqref{compatible} and \eqref{sanstorsion}, for any $u \in \M$, $\mathrm{L}_u$ is skew-symmetric with respect to $\prs$, and $\ad_u = \mathrm{L}_u - \mathrm{R}_u$, where $\ad_u: \M \to \M$ is given by $\ad_u v = [u, v]$, for all $v\in\M$.

Following the approach used for pseudo-Euclidean Lie algebras, we define an analogous curvature operator for pseudo-Euclidean Malcev algebras. The precise definition is given below. 
\ssbegin{Definition}
Let $(\M,\br, \prs)$ be a pseudo-Euclidean Malcev algebra. We define the curvature of  
$(\M,\br,\prs)$ as follows:
    \begin{equation}
K(u, v, w) = \Ll_{[[u, v], w]} - \Ll_{[u, w]} \circ \Ll_v - \Ll_u \circ \Ll_{[v, w]} + \Ll_v \circ \Ll_u \circ \Ll_w - \Ll_w \circ \Ll_v \circ \Ll_u,
\label{K} \end{equation}
for all $u,v,w\in\M.$
\end{Definition}
\ssbegin{Proposition}
 Let $(\M,\br, \prs)$ be a pseudo-Euclidean Malcev algebra. Then we have 
 $$K(u, v, w)z+K(v,w,z)u+K(w,z,u)v+K(z,u,v)w=0,$$
 for all $u,v,w,z\in \M$.
\end{Proposition}
\begin{proof}
For any $u,v,w,z\in \M$, we have 
\[
\begin{cases}
	[[u, w], [v, z]] = -[u, w] \bullet (z \bullet v) - [v, z] \bullet (u \bullet w) + [u, w] \bullet (v \bullet z) + [v, z] \bullet (w \bullet u), \\

	[[[u, v], w], z] = z \bullet (w \bullet (u \bullet v)) - z \bullet ([u, v] \bullet w) + [[u, v], w] \bullet z - z \bullet (w \bullet (v \bullet u)), \\

	[[[v, w], z], u] = -u \bullet (z \bullet (w \bullet v)) + u \bullet (z \bullet (v \bullet w)) - u \bullet ([v, w] \bullet z) + [[v, w], z] \bullet u, \\

	[[[w, z], u], v] = [[w, z], u] \bullet v - v \bullet (u \bullet (z \bullet w)) + v \bullet (u \bullet (w \bullet z)) - v \bullet ([w, z] \bullet u), \\

	[[[z, u], v], w] = -w \bullet ([z, u] \bullet v) + [[z, u], v] \bullet w - w \bullet (v \bullet (u \bullet z)) + w \bullet (v \bullet (z \bullet u)).
\end{cases}
\] 
According to Sagle's identity, we have 
\begin{align*}
 0&= [[[u, v], w], z]+  [[[v, w], z], u]+[[[w, z], u], v]+[[[z, u], v], w]-[[u, w], [v, z]] \\&= K(u, v, w)z+K(v,w,z)u+K(w,z,u)v+K(z,u,v)w.
\end{align*}
\end{proof}
\ssbegin{Definition}
 Let $(\M, \br, \prs)$ be a pseudo-Euclidean Malcev algebra. We call $(\M, \br, \prs)$ a flat pseudo-Euclidean  Malcev algebra
if $K=0$.
\end{Definition}

 \begin{Remark}
Let $(\M, \br, \prs)$ be a flat pseudo-Euclidean Lie algebra, and let $\bullet$ denote its associated Levi-Civita product. Then, it is clear that the pair $(\M, \bullet)$ forms a pre-Malcev algebra. 

Conversely, let $(\A, \bullet, \prs)$ be a pre-Malcev algebra endowed with a symmetric non-degenerate bilinear form $\prs$ such that all left multiplications of $(\A, \bullet)$ are skew-symmetric with respect to $\prs$. Then, by Prop.~\ref{Sara}, the algebra $\A^-$ is a Malcev algebra. Moreover, the identity
\[
\langle u \bullet v, w \rangle = -\langle v, u \bullet w \rangle
\]
combined with the uniqueness of the Levi-Civita product on $(\A^-, \prs)$ implies that $\bullet$ is precisely the Levi-Civita product associated with the pseudo-Euclidean Malcev algebra $(\A^-, \prs)$. Consequently, $(\A^-, \prs)$ is a flat pseudo-Euclidean Malcev algebra.
\end{Remark}

\ssbegin{Remark}
  Let $(\M,\br,\prs)$ be a flat pseudo-Euclidean Malcev algebra, let $\bullet$ denote its Levi-Civita product, and let $\Ll$ denote the left multiplication associated with $\bullet$. Then $\Ll$ is a representation of the Malcev algebra $(\M, \br)$.
 \end{Remark}

\ssbegin{Definition}
   Let \((\G, \br, \prs)\) be a pseudo-Euclidean Lie algebra. We say that \((\G, \br, \prs)\) is flat its curvature operator $K_0$, defined by 
\[
K_0(u, v)w = \Ll_{[u, v]}w - [\Ll_u, \Ll_v]w
\]
is identically zero.
\end{Definition}
\ssbegin{Definition}
Let \( (\M, \br) \) be an algebra and \( B: \M \times \M \to \mathbb{K} \) a non-degenerate symmetric bilinear form. We say that \( B \) is a quadratic structure on \( \M \) if  for all \( u, v, w \in \M \), the following condition holds:
\[
B([u,v], w) + B(v, [u,w]) = 0.
\]
In this case, the triple \( (\M, \br, B) \) is called a quadratic algebra.
\end{Definition}

\ssbegin{Proposition}$($\cite{Medina}$)$\label{medina}
  Let $(\g, \br, B)$ be a quadratic non-abelian Lie algebra. Then $(\G, \br, B)$ is flat  as a Lie algebra if and only if $(\g, \br)$ is 2-step nilpotent.  
\end{Proposition}

\ssbegin{Proposition}\label{Quadratic}
    Let $(\M, \br, \prs)$ be a quadratic Malcev algebra. Then, $(\M, \br, \prs)$ is flat if and only if $(\M, \br)$ is 3-step nilpotent. 
\end{Proposition}

\begin{proof}
We assume that \((\M, \br, \prs)\) is a quadratic Malcev algebra. According to formula \eqref{lv}, we have for any \( u, v \in \M \),
\[
\Ll_u(v) = \frac{1}{2} \ad_u(v).
\]
Thus, for any \( u, v, w \in \M \) we have

\begin{align*}
    K(u,v,w) &= \Ll_{[[u, v], w]} - \Ll_{[u, w]} \circ \Ll_v - \Ll_u \circ \Ll_{[v, w]} + \Ll_v \circ \Ll_u \circ \Ll_w - \Ll_w \circ \Ll_v \circ \Ll_u\\
    &= \frac{1}{2} \ad_{[[u, v], w]} - \frac{1}{4} \ad_{[u, w]} \circ \ad_v - \frac{1}{4} \ad_u \circ \ad_{[v, w]} + \frac{1}{8} \ad_v \circ \ad_u \circ \ad_w - \frac{1}{8} \ad_w \circ \ad_v \circ \ad_u \\
    &= \frac{3}{8} \ad_{[[u, v], w]} - \frac{1}{8} \ad_{[u, w]} \circ \ad_v - \frac{1}{8} \ad_u \circ \ad_{[v, w]}\\
    &\quad + \frac{1}{8} \Big( \ad_{[[u, v], w]} - \ad_{[u, w]} \circ \ad_v - \ad_u \circ \ad_{[v, w]} + \ad_v \circ \ad_u \circ \ad_w - \ad_w \circ \ad_v \circ \ad_u \Big)\\
    &= \frac{3}{8} \ad_{[[u, v], w]} - \frac{1}{8} \ad_{[u, w]} \circ \ad_v - \frac{1}{8} \ad_u \circ \ad_{[v, w]}.
\end{align*}

Therefore, if \((\M, \br)\) is 3-step nilpotent, then \( K(u,v,w) = 0 \), which means that  \((\M, \br, \prs)\) is flat.
Conversely, assume that \((\M, \br, \prs)\) is flat, namely \( K(u,v,w) = 0 \) for all \( u,v,w \in \M \). Then we have the identity
\[
3[[[u, v], w],z] = [[u, w], [v,z]] - [[[v, w], z],u].
\]
It follows that
\begin{align*}
    3[[[u, v], w],z] &= [[u, w], [v,z]] - [[[v, w], z],u]\\
    &= [[u, w], [v,z]] - \frac{1}{3} \Big( [[v, z], [w,u]] - [[[w, z], u],v] \Big)\\
    &= \frac{2}{3} [[u, w], [v,z]] + \frac{1}{3} [[[w,z], u],v]\\
    &= \frac{2}{3} [[u, w], [v,z]] + \frac{1}{9} [[w, u], [z,v]] - \frac{1}{9} [[[z, u], v],w] \\
    &= \frac{7}{9} [[u, w], [v,z]] - \frac{1}{9} [[[z, u], v], w]\\
    &= \frac{7}{9} [[u, w], [v,z]] - \frac{1}{27} [[z,v],[u,w]] + \frac{1}{27} [[[u, v], w],z]\\
    &= \frac{20}{27} [[u, w], [v,z]] + \frac{1}{27} [[[u, v], w],z]
\end{align*}
Consequently we obtain 
\[
[[[u, v], w],z] = \frac{1}{4} [[u,w],[v,z]] = [[[v,w], z],u],
\]
and it follows that,
\[
\ad_{[[u, v], w]} = -\ad_u \circ \ad_{[v,w]} = \ad_v \circ \ad_{[u,w]}, \quad \text{and} \quad \ad_{[[u, v], w]} = \frac{1}{4} \ad_{[u,w]} \circ \ad_v.
\]
Since \( \ad_z \) is skew-symmetric for all \( z \in \M \), we conclude that
\[
\ad_{[[u, v], w]} = -\ad_{[u,w]} \circ \ad_v.
\]
Thus, \(\ad_{[[u,v],w]} = 0\), which implies that \((\M, \br)\) is 3-step nilpotent.
\end{proof}

Since every Lie algebra is a Malcev algebra,  we have the following proposition.

\ssbegin{Proposition}\label{flmal}  
Let \((\G, \br, \prs)\) be a pseudo-Euclidean Lie algebra. If \((\G, \br, \prs)\) is flat in the classical Lie algebra curvature sense, i.e., \( K_0 = 0 \), then it is also flat as a Malcev algebra, i.e., \( K = 0 \).  
\end{Proposition}  
\begin{proof}
 Let $u,v,w\in \G$. We have 
 \begin{align*}
K(u,v,w) &= K_0([u,v],w) + [K_0(u,v), \Ll_w] + [[\Ll_u, \Ll_v], \Ll_w] - K_0(u,w) \circ \Ll_v - [\Ll_u, \Ll_w] \circ \Ll_v \\
&\quad - \Ll_u \circ K_0(v,w) - \Ll_u \circ [\Ll_v, \Ll_w] + \Ll_v \circ \Ll_u \circ \Ll_w - \Ll_w \circ \Ll_v \circ \Ll_u \\
&= K_0([u,v],w) + [K_0(u,v), \Ll_w] - K_0(u,w) \circ \Ll_v - \Ll_u \circ K_0(v,w) \\
&\quad + \Ll_u \circ \Ll_v \circ \Ll_w - \Ll_v \circ \Ll_u \circ \Ll_w - \Ll_w \circ \Ll_u \circ \Ll_v + \Ll_w \circ \Ll_v \circ \Ll_u \\
&\quad - \Ll_u \circ \Ll_w \circ \Ll_v + \Ll_w \circ \Ll_u \circ \Ll_v - \Ll_u \circ \Ll_v \circ \Ll_w + \Ll_u \circ \Ll_w \circ \Ll_v \\
&\quad + \Ll_v \circ \Ll_u \circ \Ll_w - \Ll_w \circ \Ll_v \circ \Ll_u \\
&= K_0([u,v],w) + [K_0(u,v), \Ll_w] - K_0(u,w) \circ \Ll_v - \Ll_u \circ K_0(v,w)=0.\end{align*}\end{proof}
\ssbegin{Remark}
In contrast to Proposition~\ref{flmal}, we investigate whether the flatness of a pseudo-Euclidean Lie algebra in the Malcev sense guarantees its flatness in the classical Lie algebra curvature sense. The answer is negative. Indeed, there exist pseudo-Euclidean Lie algebras that are flat as Malcev algebras but are not flat as Lie algebras.

For instance, consider a quadratic 3-step nilpotent Lie algebra $(\mathfrak{g}, \br, B)$. By Prop.~\ref{Quadratic}, $(\mathfrak{g}, \br, B)$ is flat as a Malcev algebra, meaning that the curvature operator $K$ vanishes. However, by Prop.~\ref{medina}, the same Lie algebra $(\mathfrak{g}, \br, B)$ is not flat in the classical Lie algebra curvature sense, since its corresponding curvature operator $K_0 \neq 0$. This example illustrates that the class of flat pseudo-Euclidean Malcev  algebras is strictly larger and that flatness in the Malcev sense does not, in general, imply flatness in the classical Lie algebra curvature sense.
 
\end{Remark}

Let $(\M, \br, \prs)$ be a pseudo-Euclidean Malcev algebra. Then $K=0$ is also
equivalent to one of the following relations:
\begin{equation}
\begin{aligned}\label{re1}
	&\Rr_u\circ\Rr_v\circ\Rr_w-\Rr_u\circ\Rr_v\circ\Ll_w-\Rr_u\circ\Ll_v\circ\Rr_w+\Rr_u\circ\Ll_v\circ\Ll_w-\Rr_{w\bullet u}\circ\Rr_v\\
	&+\Rr_{w\bullet u}\circ\Ll_v-\Rr_{[w, v]\bullet u}+\Ll_w\circ \Rr_{v\bullet u}-\Ll_v\circ\Ll_w\circ\Rr_u=0,
\end{aligned}
\end{equation}
   \begin{equation}
\begin{aligned}\label{re2}
     &\Rr_u\circ\Rr_v\circ\Ll_w-\Rr_u\circ\Rr_v\circ\Rr_w-\Rr_u\circ\Ll_v\circ\Ll_w-\Ll_w\circ\Rr_u\circ\Rr_v\\
     &+\Ll_w\circ\Rr_u\circ\Ll_v+\Rr_{w\bullet(v\bullet u)}-\Ll_v\circ \Rr_{w\bullet u}+\Rr_u\circ\Ll_v\circ\Rr_w-\Ll_{[w,v]}\circ\Rr_u=0       \end{aligned}
    \end{equation}
     \begin{equation}
\begin{aligned}\label{re3}
	&\Rr_u\circ\Ll_{[w,v]}-\Rr_u\circ\Rr_{[w,v]}-\Rr_{v\bullet u}\circ\Ll_w+\Rr_{v\bullet u}\circ\Rr_w-\Ll_w\circ\Rr_u\circ\Ll_v\\ 
	&+\Ll_w\circ\Rr_u\circ\Rr_v+\Ll_v\circ\Ll_w\circ\Rr_u-\Rr_{v\bullet(w\bullet u)}=0,
\end{aligned}
    \end{equation}
for all $u,v,w\in \M$. We denote by \( Z(\mathcal{M}) \) its center and by \( [\mathcal{M}, \mathcal{M}] \) its derived ideal. Here are the expressions we get:
\begin{equation}
  (\mathcal{M} \bullet \mathcal{M})^\perp = \{ u \in \mathcal{M} \mid \Rr_u = 0 \}\quad  \text{and}  \quad  
  [\mathcal{M}, \mathcal{M}]^\perp = \{ u \in \mathcal{M} \mid \Rr_u^* = \Rr_u \}.
\label{Ortho}
\end{equation}  
where \(\mathcal{M} \bullet \mathcal{M} = \operatorname{span} \{ u \bullet v \mid u, v \in \mathcal{M} \}\).

\ssbegin{Proposition}\label{nilpotent}
	Let $(\M, \br,\prs)$ be a flat pseudo-Euclidean Malcev algebra. We have the following results:
	\begin{enumerate}
    \item $\Ll_u^3=0$  for all $u\in Z(\M)$.  In particular, if $\prs$ is Euclidean, then $$Z(\M)=\{u\in\M\mid \Ll_u=\Rr_u=0\}.$$
		\item For any $u\in [\M,\M]^\bot$, $\Rr_u$ is nilpotent.
		\item If $\prs$ is Euclidean, then $\Rr_u=0$  for any $u\in [\M,\M]^\bot.$ In particular,  $\ad_u$ is skew-symmetric for any $u\in [\M,\M]^\bot.$
	\end{enumerate}
	\end{Proposition}
	
    \begin{proof}
    \label{rapphar6}
       \begin{enumerate}
       
    \item  Let \( u \in Z(\M) \). By using equation~\eqref{lv} we have \( u \bullet u = 0 \). Since $\Ll_u=\Rr_u$,  using the relation \eqref{re1}, we obtain $\Ll_u^3=0$. Since $\prs$ is Euclidean, $\Ll_u$ is skew-symmetric and nilpotent, therefore $\Ll_u=\Rr_u=0$.
   
\item Let \( u \in [\M, \M]^\bot \). By using equation~\eqref{lv}, we have \( u \bullet u = 0 \). Then using  the relation \eqref{re1}, we obtain the following system
\begin{equation}
    \Rr_u^3-\Rr_u^2\circ\Ll_u-\Rr_u\circ\Ll_u\circ\Rr_u+\Rr_u\circ\Ll_u^2-\Ll_u^2\circ\Rr_u= 0, 
    \label{re5}\end{equation}
 Since $\Rr_u$ is symmetric, then we have    
 \begin{equation}
    \Rr_u^3+\Ll_u\circ\Rr_u^2+\Rr_u\circ\Ll_u\circ\Rr_u+\Ll_u^2\circ\Rr_u-\Rr_u\circ\Ll_u^2= 0, 
    \label{re6}\end{equation}
From the relations \eqref{re5} and  \eqref{re6}, we deduce that 
\begin{equation*}
\Rr_u^3+\Ll_u\circ\Rr_u^2=-\Rr_u^3+\Rr_u^2\circ\Ll_u \end{equation*}
Accordingly  \begin{equation}
2\Rr_u^3=\Rr_u^2\circ\Ll_u -\Ll_u\circ\Rr_u^2=[\Rr_u^2, \Ll_u].\end{equation}
For any $n\in\mathbb{N}$, we have $2\Rr_u^{3+n}=\Rr_u^n\circ(\Rr_u^2\circ\Ll_u -\Ll_u\circ\Rr_u^2)$, so
\begin{align*}
2\tr(\Rr_u^{3+n})&=\tr\big(\Rr_u^n\circ(\Rr_u^2\circ\Ll_u -\Ll_u\circ\Rr_u^2)\big) \\&= \tr\big(\Rr_u^{n+2}\circ\Ll_u\big)-\tr\big(\Rr_u^n\circ\Ll_u\circ\Rr_u^2)\big)= \tr\big(\Rr_u^{n+2}\circ\Ll_u\big)-\tr\big(\Rr_u^{n+2}\circ\Ll_u)\big)\\&=0.\end{align*}
Thus we get that $\tr(\Rr_u^{3+n})=0$ for any $n\in \mathbb{N}$, which implies that \( \Rr_u \) is nilpotent.
\item By item (2), $\mathrm{R}_u$ is nilpotent for all $u\in [\M, \M]^\bot$, and by \eqref{Ortho}, $\mathrm{R}_u$ is symmetric with respect to $\prs$. Since $\prs$ is Euclidean, then $\Rr_u=0$. Consequently, $\ad_u=\Ll_u$ is skew-symmetric.
\end{enumerate}
\end{proof}
We recall an important result, known as Milnor's theorem, characterizing flat Euclidean Lie algebras:  
\ssbegin{Theorem} [\cite{Milnor}]
Let  \( (\G, \br, \prs) \) be a Euclidean Lie algebra.  Then  \( (\G, \br, \prs) \) is flat if and only if \( \G \) splits as  
$
\G = [\G, \G]^\perp \oplus [\G, \G],
$  
where \( [\G, \G]^\perp \) and \( [\G, \G] \) are abelian, and \( \mathrm{ad}_u \) is skew-symmetric for any \( u \in [\G, \G]^\bot \).  
\end{Theorem}
We are thus led to consider whether Milnor's theorem extends to flat Euclidean Malcev algebras, a generalization that would provide valuable insights into the geometric structure of Malcev algebras, that would illuminate these algebras compared to Euclidean Lie algebras. We now present a complete solution to this problem for flat Euclidean solvable Malcev algebras, deriving the conditions under which Milnor's fundamental result holds in this generalized framework.

 \ssbegin{Theorem} \label{EMalcev}
 Let \( (\M, \br, \prs) \) be a Euclidean solvable Malcev algebra.  Then \( (\M, \br, \prs) \) is flat if and only if \( \M \) splits as  
$
\M = [\M, \M]^\perp \oplus [\M, \M],
$  
where \( [\M, \M]^\perp \) and \( [\M, \M] \) are abelian, and \( \mathrm{ad}_u \) is skew-symmetric for any \( u \in [\M, \M]^\bot \). Furthermore, the dimension of $[\M,\M]$ is even.  
\end{Theorem}
\begin{proof}
Assume that \( (\M, \br, \prs) \) is a flat Euclidean solvable Malcev algebra. According to Proposition \ref{nilpotent}, we have \( \Rr_u = 0 \) for all \( u \in [\M, \M]^\bot \). It follows that for all \( v \in [\M, \M] , w\in \M\) and \( u \in [\M, \M]^\bot \), we have:

\[
\langle v \bullet w, u \rangle = -\langle w, v \bullet u \rangle = 0, \esp \langle w \bullet v, u \rangle = -\langle v, w \bullet u \rangle = 0,
\]
This implies that \( \mathcal{M} \) splits as:
$$
\mathcal{M} = [\mathcal{M}, \mathcal{M}]^\perp \oplus [\mathcal{M}, \mathcal{M}],
$$  where \( [\mathcal{M}, \mathcal{M}]^\perp \) is abelian and \( [\mathcal{M}, \mathcal{M}] \) is a two-sided ideal for the Levi-Civita product. \\It follows that \( ([\mathcal{M}, \mathcal{M}], \bullet_{\mid_{[\mathcal{M}, \mathcal{M}]}}, \prs_{\mid_{[\mathcal{M}, \mathcal{M}] \times [\mathcal{M}, \mathcal{M}]}} )\) is a flat Euclidean Malcev algebra. Since \( \mathcal{M} \) is solvable, we know that \( [\mathcal{M}, \mathcal{M}] \) is nilpotent (see \cite{Elduque}). Let \( \mathcal{M}^{'} = [\mathcal{M}, \mathcal{M}] \), and define \( \bullet^{'} = \bullet_{\mid [\mathcal{M}, \mathcal{M}]} \) for the Levi-Civita product associated to \( (\mathcal{M}^{'}, \br, \prs^{'}) \). \\
 Let \( u \in Z(\mathcal{M}^{'}) \), according to Proposition \ref{nilpotent}, we have  \( \Ll^{'}_u = 0 \) and \( \Rr^{'}_u = 0 \). Hence we have   \( Z(\mathcal{M}^{'}) \subset [\mathcal{M}^{'}, \mathcal{M}^{'}]^\perp \) (the orthogonal complement with respect to \( \prs^{'} \)), and we deduce that
$$
Z(\mathcal{M}^{'}) \cap [\mathcal{M}^{'}, \mathcal{M}^{'}] \subset Z(\mathcal{M}^{'}) \cap Z(\mathcal{M}^{'})^\perp. $$
 If \( \mathcal{M}^{'} \) is not abelian, we must have \( Z(\mathcal{M}^{'}) \cap Z(\mathcal{M}^{'})^\perp \neq \{0\} \), which contradicts the fact that \( \prs \) is Euclidean. Therefore, we conclude that
\[
\mathcal{M} = [\mathcal{M}, \mathcal{M}]^\perp \oplus [\mathcal{M}, \mathcal{M}],
\]
where \( [\mathcal{M}, \mathcal{M}]^\perp \) is abelian, \( [\mathcal{M}, \mathcal{M}] \) is abelian, and for any \( u \in [\mathcal{M}, \mathcal{M}]^\perp \), \( \ad_u = \Ll_u \) is skew-symmetric. 
Conversely, if \( \mathcal{M} \) splits as above, then it is easy to check that \( (\mathcal{M}, \br, \prs) \) is flat. 
 We have that  $$ [\mathcal{M}, \mathcal{M}]^\perp = \{ u \in \mathcal{M} \mid \Rr_u = 0 \}, $$ and for any \( v \in Z(\mathcal{M})\),  \( \Ll_v = \Rr_v = 0 \), which implies that we have \( Z(\mathcal{M}) \subset [\mathcal{M}, \mathcal{M}]^\perp \). \\ Let \( \{e_1, \dots, e_p\} \) be a basis of \( [\mathcal{M}, \mathcal{M}]^\perp \). Next we will prove that  \( \ker \, \ad_{e_1} \) has even codimension in \( [\mathcal{M}, \mathcal{M}] \). Since \( \ad_{e_1} \) and \( \ad_{e_2} \) commute, the intersection \( \ker \, \ad_{e_1} \cap \ker \, \ad_{e_2} \) has even codimension in \( \ker \, \ad_{e_1} \), and thus in $ [\mathcal{M}, \mathcal{M}] $.  
 
 It follows that the subspace \( K = [\mathcal{M}, \mathcal{M}] \cap \left( \bigcap_{i=1}^p \ker \, \ad_{e_i} \right) \) has even codimension in \( [\mathcal{M}, \mathcal{M}] \). Since \( K \subset Z(\mathcal{M}) \), we deduce that \( K = \{0\} \). Therefore, \( \dim [\mathcal{M}, \mathcal{M}] \) is even as we claimed.
\end{proof}
\ssbegin{Corollary}\label{co1}
   Let $(\M, \br,\prs)$ be a flat  Euclidean solvable Malcev algebra. Then \( (\M, \br, \prs) \) is a flat Euclidean Lie algebra, i.e., \( (\M, \br, \prs) \) is a Euclidean Lie algebra and \((\M, \br, \prs) \) is flat as a Lie algebra.
\end{Corollary}
\ssbegin{Corollary}\label{Nul}
Let $(\M, \br,\prs)$ be a flat  Euclidean nilpotent Malcev algebra. Then $(\M, \br)$ is abelian.    
\end{Corollary}
In Corollary \ref{co1},  
we have shown that any flat Euclidean solvable Malcev algebra is also flat as a Euclidean Lie algebra. This raises the following question: can this result be generalized to all flat Euclidean Malcev algebras? In other words, can we assert that any flat Euclidean Malcev algebra is both a Lie algebra and flat as a Lie algebra?
Based on our analysis, we propose the following open question:

\textbf{Open question:}
   An Euclidean Malcev algebra is flat if and only if it is a Lie algebra and flat as a Lie algebra.

\section{ Double extension of flat pseudo-Euclidean Malcev algebras }\label{s3}
This section presents the double extension construction for flat pseudo-Euclidean Malcev algebras, a process that was first introduced in the study of quadratic Lie algebras by A. Medina and P. Revoy in \cite{Medina1}. 

\subsection{Central extension of pre-Malcev algebras}
Let $(\A, \bullet)$ be a pre-Malcev algebra,  let $\h := \K e$ be a one-dimensional vector space, and let $\mu : \A \times \A \rightarrow \K$ be a bilinear map.  
We define a new product $\diamond$ on the vector space
$
\widetilde{\A} := \A \oplus \K e
$
as follows:
\[
(u+\alpha e)\diamond(v+\beta e) := u \bullet v + \mu(u,v)e,
\qquad \forall\, u,v \in \A,\; \alpha, \beta \in \K.
\]
Thus,   $(\widetilde{\A}, \diamond)$ is a  pre-Malcev algebra if and only if the following condition holds (for all $u,v,w, z\in \A$):
\begin{equation}
\mu([[u, v]_\bullet, w]_\bullet, z) = \mu([u, w]_\bullet, v \bullet z) + \mu(u, [v, w]_\bullet \bullet z) + \mu(w, v \bullet (u \bullet z)) - \mu(v, u \bullet (w \bullet z)). \label{con-left}
\end{equation}
  
In this case, $(\widetilde{\A}, \diamond)$ is called the \emph{central extension} of $(\A, \bullet)$ by means of  $\mu$.
\subsection{Almost semi-direct product of pre-Malcev algebras}

Let $(\A, \bullet)$ be a pre-Malcev algebra. Let $\mathcal{V} := \K d$ be a one-dimensional vector space, let $
\de, \xi : \A \rightarrow \A, 
$ be two endomorphisms, and let $b_0 \in \A$, and let $\lambda \in \mathbb{K}$ be a scalar.  
On the vector space 
$
\overline{\A} := \K d \oplus \A,
$
we define a new product $\bar{\bullet}$ as follows:
\[
d \bar{\bullet} d := \lambda d + b_0, \qquad
d \bar{\bullet} u := \de(u), \qquad
u \bar{\bullet} d := \xi(u), \qquad
u \bar{\bullet} v := u \bullet v,
\]
for all $u,v \in \A$.  

Then $(\overline{\A}, \bar{\bullet})$ is a pre-Malcev algebra if and only if the following conditions are satisfied (for all $u,v, w \in \A$): 
\begin{equation}
\begin{cases}
\xi([[u, v]_\bullet, w]_\bullet)
=[u,w]_\bullet \bullet \xi(v)+u\bullet\xi([v,w]_\bullet)-v\bullet(u\bullet \xi(w))+w\bullet(v\bullet \xi(u)), \\[0.4em] 
(\xi-\de)([u,v]_\bullet)\bullet w=(\xi-\de)(u)\bullet (v\bullet w)+u\bullet ((\xi-\de)(v)\bullet w)-v\bullet(u\bullet\de(w))+\de(v\bullet(u\bullet w)),\\[0.4em]
\xi((\xi-\de)([u,v]_\bullet))= (\xi-\de)(u)\bullet \xi(v)+u\bullet \xi((\xi-\de)(v))-v\bullet(u\bullet b_0)-\lambda v\bullet\xi(u)+\de(v\bullet\xi(u)),\\[0.4em] 
[(\xi-\de)(u), v]_\bullet \bullet w=(\xi-\de)(v)\bullet(u\bullet w)-\de([u,v]_\bullet \bullet w)+u\bullet \de(v\bullet w)-v\bullet(u\bullet \de(w)),\\[0.4em] [(\xi-\de)(u), v]_\bullet \bullet w=[u,v]_\bullet \bullet\de(w)-u\bullet ((\xi-\de)(v)\bullet w)-\de(u\bullet(v\bullet w))+v\bullet\de(u\bullet w)),
\\[0.4em]
\xi([(\xi-\de)(u), v]_\bullet) =(\xi-\de)(v)\bullet\xi(u)-\de(\xi([u,v]_\bullet))+u\bullet \de(\xi(v))-v\bullet(u\bullet b_0)-\lambda v\bullet \xi(u),\\[0.4em]
\xi([(\xi-\de)(u), v]_\bullet) =[u,v]_\bullet \bullet b_0+\la \xi([u,v])-u\bullet\xi(\xi-\de)(v))-\de(u\bullet\xi(v))+v\bullet\de(\xi(u)),
\\[0.4em] (\xi-\de)^2(u)\bullet v=\de((\xi-\de)(u)\bullet v)+u\bullet\de^2(v)-\de(u\bullet \de(v)), \\[0.4em] (\xi-\de)^2(u)\bullet v=(\xi-\de)(u)\bullet \de(v)-\de(u\bullet\de(v))+\de^2(u\bullet v), \\[0.4em] \xi((\xi-\de)^2(u))=\de(\xi(\xi-\de)(u))-\lambda \de(\xi(u))-\de(u\bullet b_0)+\lambda u\bullet b_0+u\bullet\de(b_0), \\[0.4em] \xi((\xi-\de)^2(u))=(\xi-\de)(u)\bullet b_0+\la\xi((\xi-\de)(u))-\de(u\bullet b_0)-\la \de(\xi(u))+\de^2(\xi(u)),\\[0.4em] (\xi-\de)(u)\bullet \de(v)+\de((\xi-\de)(u)\bullet v)+\de^2(u\bullet v)-u\bullet \de^2(v)=0,\\[0.4em](\xi-\de)(u)\bullet b_0-\la \xi((\xi-\de)(u))+\de(\xi((\xi-\de)(u)))+\de^2(\xi(u))-u\bullet \de(b_0)-\la\xi(u)=0
.
\end{cases}
\label{produit-semi}
\end{equation}
\ssbegin{Definition}
If $(\de, \xi, b_0, \lambda)$ satisfies the compatibility conditions \eqref{produit-semi},  
then the pre-Malcev algebra $(\overline{\A}, \bar{\bullet})$ is called the 
\emph{almost semi-direct product} of the pre-Malcev algebra $(\A,\bullet)$ 
by the one-dimensional vector space $\K d$ by means of $(\de, \xi, b_0, \lambda)$.  
In this case, the collection  
$
(\mathcal{V}, \A, \de, \xi, b_0, \lambda)
$
is referred to as a \emph{context of an almost semi-direct product of pre-Malcev algebras} and we say that $(\de, \xi, b_0, \lambda)$ is 
\emph{admissible on $\A$}.

\end{Definition}

\subsection{Flat  double extensions of flat pseudo-Euclidean Malcev algebras}
In this paragraph,  we introduce the notion of a {\it flat double extension} of flat pseudo-Euclidean Malcev algebras.

Let $(\M, \br_\M, \prs_\M)$ be a flat pseudo-Euclidean Malcev algebra. Denote by $\bullet$ the Levi-Civita product associated with it.  
Let $\mathcal{V} = \K d$ be a one-dimensional vector space and $\mathcal{V}^* = \K e$ be its dual.  

In order to introduce the process of flat double extension by $\mathcal{V}$, we consider  linear maps $
\xi, D: \M \rightarrow \M,
$ an element $b_0 \in \M$, and a scalar $\lambda \in \K$.  
Set $\delta := \xi + D$. We extend the linear maps $\delta$ and $\xi$ to 
$
\widetilde{\de}, \widetilde{\xi} : \M \oplus \mathcal{V}^* \rightarrow \M \oplus \mathcal{V}^*,
$
as follows:   
\begin{equation}
\label{eq:double-tilde}
\begin{array}{c}
\widetilde{\de}(u + \alpha e)
:= (\xi+D)(u) + ( \langle b_0, u\rangle_\M+\la)\, e, \quad   
\widetilde{\xi}(u + \alpha e)
:= \xi(u)\,.
\end{array}
\end{equation}
for all $u \in \M$ and $\alpha \in \K$. We also choose $\widetilde{b}_0
:=- b_0 \in \M$.

\medskip

\sssbegin{Lemma} \label{Lemma1}
The space $\mathcal{H} := \M \oplus \mathcal{V}^*$ equipped with the product
\[
(u + \alpha e) \diamond (v + \beta e)
:= u \star v - \langle \xi(u), v\rangle_\M\, e,
\qquad \forall\, u,v \in \M,\; \alpha, \beta \in \K,
\]
is a pre-Malcev algebra if and only if
\begin{equation}
\xi([[u, v]_\M, w]_\M)
=[u,w]_\M \bullet \xi(v)+u\bullet\xi([v,w]_\M)-v\bullet(u\bullet \xi(w))+w\bullet(v\bullet \xi(u)),
\label{eq:xi-condition0}
\end{equation}
for every $u,v, w \in \M$.
\end{Lemma} 

\begin{proof} 
The algebra $(\mathcal{H}, \diamond)$ is pre-Malcev  if and only if the bilinear map 
\[
\mu(u,v) := - \langle \xi(u), v\rangle_\M\, e
\]
satisfies Eq. \eqref{con-left}, which is equivalent to Eq.  \eqref{eq:xi-condition0}. 
\end{proof}

Now, let us also assume that $(\mathcal{V}, \M, D+\xi, \xi, - b_0,-\lambda)$  is a context of almost semi-direct product 
of pre-Malcev algebras.

\medskip

\begin{Lemma}\label{Lemma2}
$(V, \mathcal{H} :=\M \oplus V^*, \widetilde{\delta}, \widetilde{\xi}, \widetilde{b_0}, -\lambda)$ 
is a context of almost semi-direct product of pre-Malcev algebras if and only if, 
for all $u,v \in\M$, the following system holds:
\begin{equation}\label{eq:claim2}
\left\{
\begin{array}{l}  
\ad_u^*\circ\xi^*\circ\xi(v)
+\ad_v^*\circ\xi^*\circ\xi(u)
+\xi(u\bullet\xi(v))
+\xi(v\bullet\xi(u))=0,
\\[0.4em] 
\xi(D([u,v])=(\xi+D)(u\bullet\xi(v))-v\bullet(\xi(D(u))-D(v)\bullet\xi(u),    \\[0.4em]  \xi(D([u,v])=(\xi+D)(u\bullet\xi(v))-u\bullet\xi(D(v))-v\bullet(\xi+D)(\xi(u))+[u,v]\bullet b_0+\la \xi([u,v]),    \\[0.4em] \xi([D(u),v])=D(v)\bullet\xi(u)+v\bullet(u\bullet b_0)-u\bullet (D+\xi)(\xi(v))-\la v\bullet \xi(u),  \\[0.4em]\xi^*\circ \xi\circ D+D^*\circ \xi^*\circ \xi-\xi^*\circ\Rr_{b_0}-\Rr_{b_0}^*\circ \xi-2\la \xi^*\circ \xi=0,
 \\[0.4em] D^*\circ\xi^*\circ \xi+\xi^*\circ(D+\xi)\circ\xi-\xi^*\circ\Rr_{b_0}+\Rr_{\xi^*(b_0)}-\Rr_{\xi^*(b_0)}^*-\la\xi^*\circ\xi=0,  \\[0.4em] (D+\xi)^2\circ\xi-\xi\circ D^2+ (D+\xi)\circ\Rr_{b_0}-\Rr_{b_0}\circ D+\la (D+\xi)\circ\xi-\la\xi\circ D=0, \\[0.4em] (D+\xi)\circ\xi\circ D-\xi\circ D^2+(D+\xi)\circ\Rr_{b_0}+\Rr_{(D+\xi)(b_0)}-\la(D+\xi)\circ \xi-\la^2\xi-\la\Rr_{b_0}=0,\\[0.4em] \xi^*\circ(D+\xi)(b_0)+D^*\circ \xi^*(b_0)-2\la \xi^*(b_0)-\Rr^*_{b_0}(b_0)=0,\\[0.4em](D+\xi)\circ \xi\circ D-\xi\circ(D+\xi)\circ D-\la \xi\circ D -\Rr_{b_0}\circ D+\la^2\xi+  \Rr_{b_0}    +\la \Rr_{(D+\xi)(b_0)}=0,  
\\[0.4em]
2\xi^*((D+\xi)(b_0))=\la b_0.
\end{array}
\right.
\end{equation}
\end{Lemma}
\begin{proof}
 Let $u, v, w \in \M$ and $\alpha, \beta \in \mathbb{K}$. 
A direct computation shows that the first identity in the system \eqref{produit-semi}
is equivalent to
\[
\xi([[u, v]_\M, w]_\M)
=[u,w]_\M \bullet \xi(v)+u\bullet\xi([v,w]_\M)-v\bullet(u\bullet \xi(w))+w\bullet(v\bullet \xi(u)),\] 
and  
$$\ad_u^*\circ \xi^*\circ\xi(v)+\ad_v^*\circ \xi^*\circ\xi(u)+\xi(u\bullet\xi(v))+\xi(v\bullet\xi(u))=0.$$

Moreover, the second relation in the system \eqref{produit-semi} is equivalent to
\[
D([u,v])\bullet w=D(u)\bullet(v\bullet w)+u\bullet(D(v)\bullet w)+v\bullet(u\bullet(\xi+D)(w))-(\xi+D)(v\bullet (u\bullet w))+u\bullet(v\bullet b_0),
\]
and
\[
\xi(D([u,v]))=(\xi+D)(u\bullet\xi(v))-v\bullet(\xi(D(u))-D(v)\bullet\xi(u),
\]

Furthermore, the third relation in the system \eqref{produit-semi} is equivalent to
\[
\xi(D([u,v]))=D(u)\bullet \xi(v)+u\bullet(\xi(D(v))-(\xi+D)(v\bullet \xi(u))-v\bullet(u\bullet b_0))-\la v\bullet\xi(u),
\]
and
\[
\xi^*\circ \xi\circ D+D^*\circ \xi^*\circ \xi-\xi^*\circ\Rr_{b_0}-\Rr_{b_0}^*\circ \xi-2\la \xi^*\circ \xi=0.
\]

Similarly, the fourth relation in the system \eqref{produit-semi} is equivalent to
\[[D(u),v]\bullet w=D(v)\bullet(u\bullet w)+(D+\xi)([u,v]\bullet w)-u\bullet(D+\xi)(v\bullet w)+v\bullet(u\bullet(D+\xi)(w))
\]
and 
\[\xi(D([u,v]))=(\xi+D)(u\bullet\xi(v))-u\bullet\xi(D(v))-v\bullet(\xi+D)(\xi(u))+[u,v]\bullet b_0+\la \xi([u,v])\]

Likewise, the fifth relation in the system \eqref{produit-semi} is equivalent to
\[[D(u),v]\bullet w=(D+\xi)(u\bullet(v\bullet w))-[u,v]\bullet (D+\xi)(w)-u\bullet(D(v)\bullet w)-v\bullet(D+\xi)(u\bullet w),
\]
and 
\[\xi([D(u),v])=D(v)\bullet\xi(u)+v\bullet(u\bullet b_0)-u\bullet (D+\xi)(\xi(v))-\la v\bullet \xi(u)\]

Moreover, the sixth relation in the system \eqref{produit-semi} is equivalent to
\[\xi([D(u),v])=D(v)\bullet \xi(u)+(D+\xi)(\xi([u,v])-u\bullet(D+\xi)(\xi(v))+v\bullet(u\bullet b_0)-\la v\bullet\xi(u),
\]
and 
\[D^*\circ\xi^*\circ \xi+\xi^*\circ(D+\xi)\circ\xi-\xi^*\circ\Rr_{b_0}+\Rr_{\xi^*(b_0)}-\Rr_{\xi^*(b_0)}^*-\la\xi^*\circ\xi=0.\]

In the same way, the seventh relation in the system \eqref{produit-semi} is equivalent to
\[\xi([D(u),v])=[u,v]\bullet b_0+(D+\xi)(u\bullet\xi(v))-u\bullet\xi(D(v))-v\bullet(D+\xi)(\xi(u))+\la\xi([u,v]),
\]
and 
\[D^*\circ\xi^*\circ \xi+\xi^*\circ(D+\xi)\circ\xi-\xi^*\circ\Rr_{b_0}+\Rr_{\xi^*(b_0)}-\Rr_{\xi^*(b_0)}^*-\la\xi^*\circ\xi=0.\]

Moreover, the eighth relation in the system \eqref{produit-semi} is equivalent to
\[D^2(u)\bullet v=u\bullet(D+\xi)^2(v)-(D+\xi)(D(u)\bullet v)-(D+\xi)(u\bullet (D+\xi)(v)),
\]
and 
\[(D+\xi)^2\circ\xi-\xi\circ D^2+ (D+\xi)\circ\Rr_{b_0}-\Rr_{b_0}\circ D+\la (D+\xi)\circ\xi-\la\xi\circ D=0.\]

The ninth relation in the system \eqref{produit-semi} is equivalent to
\[D^2(u)\bullet v=(D+\xi)^2(u\bullet v)-(D+\xi)(u\bullet (D+\xi)(v))-D(u)\bullet (D+\xi)(v),
\]
and 
\[(D+\xi)\circ\xi\circ D-\xi\circ D^2+(D+\xi)\circ\Rr_{b_0}+\Rr_{(D+\xi)(b_0)}-\la(D+\xi)\circ \xi-\la^2\xi-\la\Rr_{b_0}=0.\]

The tenth relation in the system \eqref{produit-semi} is equivalent to
\[\xi(D^2(u))=(D+\xi)(u\bullet b_0)-(D+\xi)(\xi(D(u))+u\bullet(D+\xi)(b_0)+\la(D+\xi)(\xi(u))-\la u\bullet b_0,
\]
and 
\[\xi^*\circ(D+\xi)(b_0)+D^*\circ \xi^*(b_0)-2\la \xi^*(b_0)-\Rr^*_{b_0}(b_0)=0.\]

The eleventh relation in the system \eqref{produit-semi} is equivalent to
\[\xi(D^2(u))=(D+\xi)(u\bullet b_0)+D(u)\bullet b_0+(D+\xi)^2(\xi(u))+\la \xi\circ D(u))+\la(D+\xi)(\xi(u)),
\]
and 
\[\xi^*\circ(D+\xi)(b_0)+D^*\circ \xi^*(b_0)-2\la \xi^*(b_0)-\Rr^*_{b_0}(b_0)=0.\]

Finally, the twelfth and thirteenth relations in the system \eqref{produit-semi} yield
\[D(u)\bullet(D+\xi)(v)+(D+\xi)(D(u)\bullet v)+u\bullet(D+\xi)^2(v)-(D+\xi)^2(u\bullet v)=0,
\]
and 
\[(D+\xi)\circ \xi\circ D-\xi\circ(D+\xi)\circ D-\la \xi\circ D -\Rr_{b_0}\circ D+\la^2\xi+  \Rr_{b_0}    +\la \Rr_{(D+\xi)(b_0)}=0,\]
together with
\[(D+\xi)\circ \xi\circ D-(D+\xi)^2\circ\xi+\la\xi\circ D-\la\xi\Rr_{b_0}\circ D-\Rr_{(D+\xi)(b_0)}=0,
\]
and 
\[2\xi^*((D+\xi)(b_0))=\la b_0.\]

Since  $(\mathcal{V}, D+\xi, \xi, - b_0, -\al)$ is a context of almost semi-direct product of pre-Malcev algebras, it follows that 
$(\mathcal{V}, \mathcal{H} := \M \oplus \mathcal{V}^*, \widetilde{\de}, \widetilde{\xi}, \widetilde{b}_0, -\lambda)$ is also a context of almost semi-direct product  of pre-Malcev algebras if and only if the conditions \eqref{eq:claim2} are satisfied. This proves the lemma.
\end{proof}

We now define the  flat double extension of a flat pseudo-Euclidean Malcev  algebra by a one-dimensional vector space.

\sssbegin{Theorem}\label{double-ex1}
Let $(\M, \br_{\M}, \prs_\M)$ be a flat pseudo-Euclidean Malcev algebra, and let $\bullet$ be  the Levi-Civita  product associated with it. 
Let $\mathcal{V} = \mathbb{K}d$ be a one-dimensional vector space and $\mathcal{V}^* = \mathbb{K}e$ be its dual. Assume that there exist linear maps $D, \xi : \M \to \M$, an element $b_0 \in \M$, 
and a scalar $\lambda \in \K$. 
Suppose that $(D+\xi, \xi, -b_0, -\lambda)$ is admissible on $\M$ and that 
$(D, \xi, b_0, \lambda)$ satisfies the system~\eqref{eq:claim2}. Define 
$
\widetilde{\M} := \mathbb{K}d \oplus\M \oplus \mathbb{K}e,
$
equipped with the  bracket and  bilinear form $\prs$ as follows:
\begin{equation}\label{crochet}
[d,e]=\lambda e, \quad [d,u] = D(u) + \langle
b_0, u\rangle_\M e, \quad\text{ and }\quad
[u,v] = [u,v]_{\M} -\langle
(\xi-\xi^*)(u), v\rangle_\M e,
\end{equation}
for all $u,v \in \M$ and
$$
\prs|_{\M \times \M} = \prs_{\M},
\qquad \langle e,d\rangle = \langle d,e\rangle = 1,
\qquad \langle d,\M\rangle = \langle e,\M\rangle = 0.
$$
Then, $(\widetilde{\M}, \br, \prs)$ is a flat pseudo-Euclidean Malcev algebra.

Furthermore,  the Levi-Civita  product $\star$  associated with $(\widetilde{\M}, \br, \prs)$ ,  is given by
\begin{equation}\label{Produit1-13}
\begin{array}{lll} \displaystyle 
e \star u = u \star e=e\star d =0,& e \star d = \la  e,& \displaystyle  
d \star u = (D+\xi)(u) + \langle b_0, u\rangle_\M e, \\[0.3em]
\displaystyle  u \star d = \xi(u), &d \star d = - b_0-\la d, & 
u \star v = u \bullet v - \langle \xi(u), v\rangle_\M e, 
\end{array}
\end{equation}
for all $ u,v \in\M$.
\end{Theorem}
The flat  pseudo-Euclidean Malcev algebra $(\widetilde{\M}, \br, \prs)$ is called the flat double extension of $(\M, \br_{\M}, \prs_\M)$ by means of $(D, \xi, b_0, \la).$
\begin{proof}
According to \eqref{LV}, we derive  the Levi-Civita product given by \eqref{Produit1-13}. Hence
\[
(u+\al e) \star (v+\beta e) = u \bullet v - \langle \xi(u), u\rangle_\M e,
\]
for all elements $u,v \in\M$ and $\alpha, \beta \in \mathbb{K}$.
 By Lemma \ref{Lemma1}, the first condition of system~\eqref{eq:claim2}, 
implies that the above product defines on the  vector space  $\mathcal{H} =\M \oplus \mathbb{K}e$ a pre-Malcev structure, which is a central extension of $\M$ by $\mathcal{V}$ by means of $\mu$, where 
$\mu(u,v) = - \langle \xi(u), u\rangle_\M e
$, for all $u,v\in\M$.

Now,  since $(D, \xi, b_0, \la)$ satisfies the system \eqref{eq:claim2}, it follows from Lemma \ref{Lemma2}, that 
$
(V, \mathcal{H} :=\M \oplus V^*, \widetilde{\de}, \widetilde{\xi}, \widetilde{b_0}, -\la)
$
is a context of an almost semi-direct product of pre-Malcev algebras. 
Hence, $(\M, \star)$ is a pre-Malcev algebra, 
which implies that $(\widetilde{\M}, \br, \prs)$ is a flat pseudo-Euclidean Malcev algebra.
\end{proof}
Let $(\mathcal{M}, \br, \prs)$ be a double extension of a  pseudo-Euclidean 
abelian Lie algebra $(\overline{\mathcal{M}}, \prs_{\overline{\mathcal{M}}})$ 
by means of $(\xi, D, b_0, 0)$. 
Then, the compatibility relations in the systems~\eqref{produit-semi} 
and~\eqref{eq:claim2} reduce to the following system:
\begin{equation}\label{conflat1}
\begin{cases}
\xi^{*}\circ \xi\circ D 
= \xi^{*}\circ (D+\xi)\circ \xi 
= - D^{*}\circ \xi^{*}\circ \xi 
, \\[0.3em]
\xi\circ D^{2} 
= (D+\xi)^{2}\circ \xi 
= (D+\xi)\circ \xi\circ D, \\[0.3em]
\xi^{*}((D+\xi)(b_0)) 
= D^{*}(\xi(b_0)) 
= 0.
\end{cases}
\end{equation}

\sssbegin{Example}
    Let \((\overline{\mathcal{M}} = \{e_1, e_2, e_3, e_4\}, \prs)\) be a four-dimensional abelian Lie algebra endowed with the Lorentzian metric $\prs$ defined by  
    $
    \langle e_1, e_4 \rangle = \langle e_2, e_2 \rangle = \langle e_3, e_3 \rangle = 1.$
      Let $D$ and $\xi$ be endomorphisms of $\overline{\mathcal{M}}$ and let $b_0=0$, defined as follows:
    \[
    \xi =-D= \begin{bmatrix}
    0 & 0 & a & b \\
    0 & 0 & 0 & c \\
    0 & 0 & 0 & f \\
    0 & 0 & 0 & 0
    \end{bmatrix}, 
    \]
   where $af\neq 0$ and $a,b,c,f\in \mathbb{K}$. Therefore, through a straightforward computation, we verify that \((\xi, D, b_0, 0)\) satisfies the system \eqref{conflat1}. Consequently, \(\mathcal{M} = \mathbb{K}e \oplus \overline{\mathcal{M}} \oplus \mathbb{K}d\) is a flat  double extension of the flat Lorentzian abelian Lie algebra \((\overline{\mathcal{M}}, \prs_{\overline{\mathcal{M}}})\) by means of  \((D, -\xi, b_0, 0)\).
 The bracket on \(\mathcal{M}\) is given by:   
    \[
    [d, e_3] = -a e_1, \quad [d, e_4] = -b e_1 - c e_2 - f e_3, \quad [e_2, e_4] = c e, \quad [e_3, e_4] = (f - a) e.
    \]
    Hence, \((\mathcal{M}, \br, \prs)\) is a pseudo-Euclidean Lie algebra and is flat as a Malcev algebra. However, since $af\neq 0$, it is not flat as a Lie algebra.
Indeed, using the Levi-Civita product (see \eqref{Produit1-13}), we obtain:
\[
e_3\bullet d = ae_1,\qquad
e_4\bullet d = be_1 + ce_2 + fe_3,
\quad
e_3\bullet e_4 = -a e,\qquad
e_4\bullet e_2 = -c e,\qquad
e_4\bullet e_3 = -f e.
\]
Therefore,
\[K_0(d,e_4,e_4)
= \Ll_{[d,e_4]}(e_4)-[\Ll_d,\Ll_{e_4}](e_4)
= af e\neq 0.
\]
Thus $(\mathcal{M},\br,\prs)$ is not flat as a Lie algebra.
\end{Example}

\sssbegin{Example}
    Let \((\overline{\mathcal{M}} = \{e_1, e_2, e_3, e_4\}, \prs)\) be a four-dimensional abelian Lie algebra endowed with the Lorentzian metric $\prs$ defined by  
    $
    \langle e_1, e_2 \rangle = \langle e_3, e_3 \rangle = \langle e_4, e_4 \rangle = 1.$
      Let $D$ and $\xi$ be endomorphisms of $\overline{\mathcal{M}}$ and let $b_0=0$, where
 defined as follows:  
    \[
    \xi =-D= \begin{bmatrix}
    0 & 0 & a & b \\
    0 & 0 & 0 & 0 \\
    0 & 0 & 0 & f \\
    0 & 0 & 0 & 0
    \end{bmatrix}, 
    \]
   where $af\neq 0$ and $a,b,c,f\in \mathbb{K}$. Therfore, through a straightforward computation, we verify that \((\xi, D, b_0, 0)\) satisfies the system \eqref{conflat1}. Consequently, \(\mathcal{M} = \mathbb{K}e \oplus \overline{\mathcal{M}} \oplus \mathbb{K}d\) is a flat  double extension of the flat Lorentzian abelian Lie algebra \((\overline{\mathcal{M}}, \prs_{\overline{\mathcal{M}}})\) by means of  \((D, -\xi, b_0, 0)\).
 The bracket on \(\mathcal{M}\) is given by:   
    \[
    [d, e_3] = -a e_1, \quad [d, e_4] = -b e_1 - f e_3,\quad[e_2, e_3]= a e\quad [e_2, e_4] = b e, \quad [e_3, e_4] = f e.
    \]
    Hence, \((\mathcal{M}, \br, \prs)\) is a flat pseudo-Euclidean Malcev  algebra. However, since $af\neq 0$, it is not Lie algebra.
Indeed, 
\[[[d, e_4], e_2]+ [[ e_4, e_2], d]+[[e_2, d], e_4]=af e .\]
Thus $(\mathcal{M},\br,\prs)$ is not  a Lie algebra.
\end{Example}

\sssbegin{Proposition}  \label{reduit}
Let \((\M, \br,\prs)\) be a flat pseudo-Euclidean Lie algebra. If \( J \) is a totally isotropic two-sided ideal of dimension $1$. Then we have
\begin{enumerate}
    \item  \( J \) has a trivial product (\(J\bullet J = 0\)), \( J \bullet J^\perp = 0 \), and \( J^\perp \) is a left ideal.  
\item  \( J^\perp \) is a right ideal if and only if \( J^\perp \bullet J = 0 \).  
\item  If \( J^\perp \) is a two-sided ideal of \( \M \), then the quotient Lie algebra \( \overline{\M} = J^\perp / J \) admits a canonical flat pseudo-Euclidean Malcev algebra.  
   \end{enumerate}
\end{Proposition}

\begin{proof}
\begin{enumerate}
    \item For \( u \in J \), \( v \in J^\bot \), and \( w \in \mathcal{M} \), the fact that \( J \) is  a one-dimensional totally isotropic two-sided ideal of $(\M, \bullet)$ implies that,
   \[
   \langle u \bullet v, w \rangle = -\langle v, u \bullet w \rangle = 0.
   \]  
   Therefore \( J \bullet J^\bot = 0 \). In particular, \( J \bullet J = 0 \). Additionally the relation given by 
   \[
   \langle w \bullet v, u \rangle = -\langle v, w \bullet u \rangle = 0,
   \]  
   implies that \( \mathcal{M} \bullet J^\bot \subset J^\bot \).  

\item  For \( u \in J \), \( v \in J^\bot \), and \( w \in \mathcal{M} \), we have 
   \[
   \langle v \bullet w, u \rangle = 0 \iff \langle w, v \bullet u \rangle = 0,
   \]  
   which means that \( J^\bot \bullet \mathcal{M} \subset J^\bot \) if and only if \( J^\bot \bullet J = 0 \).  

\item  Suppose that  \( J^\bot \) is a two-sided ideal of \((\mathcal{M}, \bullet)\). In this case, the quotient vector space given by   \( \overline{\mathcal{M}} = J^\bot / J \) has the structure of a Malcev algebra defined by  
   \[
   [u + J, v + J]_{\overline{M}} = [u, v] + J = (u \bullet v - v \bullet u) + J, \quad u, v \in J^\bot.
   \]  
   Since \( \prs\) is a non-degenerate, it induces on \( J^\bot \) a bilinear form with radical \( J \). \\
    Thus \( \prs|_{J^\bot \times J^\bot} \) defines by passing to the quotient a non-degenerate 
symmetric bilinear form \( \prs_{\overline{\M}} \) on the quotient vector space \( \overline{\mathcal{M}} = J^\bot / J \). This \( \prs_{\overline{\M}} \) is also pseudo-Euclidean of the Malcev algebra \( \overline{\mathcal{M}} \), given by  
   \[
   \langle u + J, v + J \rangle_{\overline{\M}} = \langle u, v \rangle, \quad \forall u, v \in J^\bot.
   \]  

If we denote the class of \( u \in J^\bot \) modulo \( J \) by \( \overline{u} = u + J \), then the Levi-Civita product  of $(\overline{\M}, \br_{\overline{M}},  \prs_{\overline{M}})$ is given by  
   \[
   \overline{u} \bullet \overline{v} = (u + J) \bullet (v + J) = u \bullet v + J,
   \]   
Therefore,  \((\overline{\mathcal{M}},  \br_{\overline{\M}}, \prs_{\overline{\M}})\) is a flat pseudo-Euclidean Malcev algebra.  
\end{enumerate}
\end{proof}
\begin{Definition}
    The Malcev algebra $\overline{\M}$ of Proposition \ref{reduit} will be called the flat pseudo-Euclidean Malcev algebra deduced from
$J^\bot$ by means of $J$. 
\end{Definition}

We now prove the converse of Theorem \ref{double-ex1}

\ssbegin{Theorem}\label{dbex}
Let $(\M, \br, \prs)$ be a flat pseudo-Euclidean Malcev algebra.  
Assume that $J$ is a one-dimensional totally isotropic two-sided ideal of $(\M, \bullet)$ and that $J^\perp$ is also a two-sided ideal of $(\M, \bullet)$.  
Denote by $\overline{\M} = J^\perp / J$ the flat pseudo-Euclidean Malcev algebra induced from $J^\perp$ by means of $J$.  
Then $(\M, \br, \prs)$ is a flat double extension of the flat pseudo-Euclidean Malcev algebra $(\overline{\M}, \br_{\overline{\M}}, \prs_{\overline{\M}})$ by means of $(D, \xi, b_0, \lambda)$.
\end{Theorem}

\begin{proof}
Let $J = \K e$ be a one-dimensional totally isotropic two-sided ideal of $(\M, \bullet)$ such that $J^\perp$ is also a two-sided ideal of $(\M, \bullet)$.  
Since $\prs$ is non-degenerate, there exists $d \in \M \setminus \{0\}$ such that 
$
\langle e, d \rangle = \langle d, e \rangle = 1.
$

As $J \subseteq J^\perp$, there exists a subspace $\mathfrak{B}$ such that $J^\perp = J \oplus \mathfrak{B}$, with the restriction $\prs_{\mathfrak{B}} = \prs|_{\mathfrak{B}\times \mathfrak{B}}$ non-degenerate.  
We then write
$
\M = \K e \oplus \mathfrak{B} \oplus \K d.
$
Since $J^\perp = J \oplus \mathfrak{B}$ is a two-sided ideal of $(\M, \bullet)$, for any $u,v \in \mathfrak{B}$, we have
\[
u \bullet v = u \bullet_\mathfrak{B} v + \mu(u,v) e,
\]
where $\mu : \mathfrak{B} \times \mathfrak{B} \to \K$ is bilinear and $\bullet_\mathfrak{B} : \mathfrak{B} \times \mathfrak{B} \to \mathfrak{B}$ is a bilinear map.  

Since $(\M, \bullet)$ is a pre-Malcev algebra, $(\mathfrak{B}, \bullet_\mathfrak{B})$ is a pre-Malcev algebra, $\mu$ satisfies the compatibility condition \eqref{con-left}, and
\[
\langle u \bullet_\mathfrak{B} v, w \rangle_\mathfrak{B} = - \langle v, u \bullet_\mathfrak{B} w \rangle_\mathfrak{B}, \quad \forall u,v,w \in \mathfrak{B}.
\]

Let $\pi : \mathfrak{B} \to J^\perp / J = \overline{\M}$, $\pi(u) = \bar{u}$, be the canonical projection.  
It is an isomorphism of pre-Malcev algebras, so we can identify $\mathfrak{B} \cong \overline{\M}$.  
Hence, we may write
$
\M = \K e \oplus \overline{\M} \oplus \K d,
$
and  $(\overline{\M}, \br_{\overline{\M}}, \prs_{\overline{\M}})$ is a flat pseudo-Euclidean Malcev algebra.

Since $J$ and $J^\perp = \K e \oplus \overline{\M}$ are two-sided ideals of $(\M, \bullet)$, they are also ideals of $(\M, \br)$.  
Then, for all $u,v \in \overline{\M}$, the brackets on $\M$ are
\[
[d,e] = \lambda e, \quad [d,u] = D(u) + T(u)e, \quad [u,v] = [u,v]_{\overline{\M}} + (\mu(u,v)+\mu(v,u)) e,
\]
where $D \in \mathrm{End}(\overline{\M})$ and $T: \overline{\M} \to \K$ is linear.  

As $\prs_{\overline{\M}}$ is non-degenerate, there exists $b_0 \in \overline{\M}$ such that $T(u) = \langle b_0, u \rangle_{\overline{\M}}$.  
Moreover, from Proposition \ref{reduit}, we have $J\bullet J^\bot=J^\bot\bullet J=\{0\}$, then the Levi-Civita product associated with $(\M, [\, ,\,], \prs)$ is
\[
\begin{aligned}
u \bullet v &= u \bullet_{\overline{\M}} v + \mu(u,v) e, & u \bullet d &= \xi(u) + f(u) e, \\
d \bullet u &= p(u) + g(u)e, & d \bullet e &= \la e, \\
d \bullet d &= \al d + c_0 + t e, & e \bullet d &= s e, \\
e \bullet u &= u \bullet e = e \bullet e = 0,
\end{aligned}
\]
where $\xi, p \in \mathrm{End}(\overline{\M})$, $f,g\in  \overline{\M}^*$, $c_0 \in \overline{\M}$, and $\la, \alpha, t, s \in \K$.

Using Koszul’s formula \eqref{lv}, for any $u,v \in \overline{\M}$ we get
\[
\mu(u,v) = - \langle \xi(u), v \rangle_{\overline{\M}}, \quad p(u) = (\xi + D)(u), \quad g(u) = \langle b_0, u \rangle, \quad f(u) = 0,\]
\[ c_0 = -b_0, \quad \la = -\al, \quad t=s=0.
\]
Thus, the Lie bracket and associated product simplify to
\[
[d,e] = \lambda e, \quad [d,u] = D(u) + \langle b_0, u \rangle_{\overline{\M}} e, \quad [u,v] = [u,v]_{\overline{\M}} - \langle (\xi - \xi^*)(u), v \rangle_{\overline{\M}} e,
\]
and
\[
\begin{aligned}
u \bullet v &= u \bullet_{\overline{\M}} v - \langle \xi(u), v \rangle_{\overline{\M}} e, & u \bullet d &= \xi(u), \\
d \bullet u &= (D + \xi)(u) + \langle b_0, u \rangle_{\overline{\M}} e, & d \bullet e &= \lambda e, \\
d \bullet d &= -\lambda d - b_0, & e \bullet d &= 0, \\
e \bullet u &= u \bullet e = e \bullet e = 0.
\end{aligned}
\]

Finally, $(\M, \bullet)$ is a pre-Malcev algebra if and only if $(D, \xi, b_0, \lambda)$ satisfies the system \eqref{eq:claim2}.  
Hence, $(\M, [\, ,\,], \prs)$ is a flat double extension of the flat pseudo-Euclidean Malcev algebra $(\overline{\M}, [\, ,\,]_{\overline{\M}}, \prs_{\overline{\M}})$ by means of  $(D, \xi, b_0, \lambda)$.
\end{proof}
\section{Flat  Lorentzian Malcev algebras}\label{s4}
In this section, we prove that all flat Lorentzian solvable Malcev algebras with a degenerate center
can be obtained by the double extension process of a flat Euclidean Lie algebra. Additionally, we prove that all  flat Lorentzian nilpotent Malcev algebras can be obtained by the double extension process of a Euclidean abelian Lie algebra. Furthermore, we show that all flat Lorentzian
 nilpotent Malcev algebras are flat Lorentzian nilpotent Lie algebras i.e $K_0=0$. 
\ssbegin{Lemma}$($\cite{Boucetta1}$)$\label{nul}
 Let $(E, \prs)$ be a Lorentzian vector space and $A$ a skew-symmetric endomorphism of $(E, \prs)$. If $A^2=0$ then $A=0$.\end{Lemma}
\ssbegin{Lemma}\label{Leprinc}
Let $(\M, \br, \prs)$    be a flat  Lorentzian Malcev algebra such that $Z(\M)$ is degenerate. Then  for any $a\in Z(\M)\cap Z(\M)^\bot$, we have  $\Ll_a=\Rr_a=0$.
\end{Lemma}

\begin{proof}
According to Proposition \ref{nilpotent}, for any \( a \in Z(\M) \), we have \( \Ll_a^3 = 0 \). If there exists $a\in Z(\M) $ such that \( \Ll_a^2 = 0 \), then using  Lemma \ref{nul}, it follows that \( \Ll_a = 0 \) for all \( a \in Z(\M) \). 

Now, suppose $\Ll_a^2 \neq 0$ for any $a \in Z(\M)$. Define \( N = \bigcap_{b \in Z(\M)} \ker (\Ll_b^2) \). Since \( \Ll_b \) is skew-symmetric, we have \( N^\perp = \sum_{b \in Z(\M)} \mathrm{Im} (\Ll_b^2) \). Since \( \Ll_a^3 = 0 \) for any \( a \in Z(\M) \), it follows that \( N^\bot \subset N \). From Koszul's formula \eqref{lv}, we obtain that  \( a \bullet b = 0 \) for all \( a, b \in Z(\M) \). It follows that \( Z(\M) \subset N \). On the other hand, using equation \eqref{K}, we take \( v = w = a \) and \( u=z \in \M \), we obtain 
\[ \Ll_a \circ \Ll_z \circ \Ll_a = \Ll_a \circ \Ll_a \circ \Ll_z. \]
From relation \eqref{re1}, we take \( u = v = a \) and \( w = z \in \M \), we get \( \Ll_a \circ \Ll_z \circ \Ll_a = 0 \). Since \( \Ll_u \) is skew-symmetric for any \( u \in \M \), it follows that \( \Ll_a \circ \Ll_a \circ \Ll_z = \Ll_z\circ \Ll_a\circ\Ll_a=0 \). Therefore we have that \( \Rr_{a \bullet (a \bullet u)} = 0 \) for all \( u \in \M \) and hence $\Rr_z=0$, for all $z\in N^\bot$.  Moreover, since \( Z(\M) \cap Z(\M)^\perp \subset N \) and \( N^\bot  / N \) is Euclidean, we deduce that \( Z(\M) \cap Z(\M)^\perp \subset N^\perp \). Consequently, \( \Rr_a = \Ll_a = 0 \) for all \( a \in Z(\M) \cap Z(\M)^\perp \).
\end{proof}

\ssbegin{Theorem}\label{Thpr}
 Let $(\M, \br, \prs)$    be a flat  Lorentzian Malcev algebra such that $Z(\M)$ is degenerate. Then $(\M, \br, \prs)$    is a flat  double extension of a flat Euclidean Malcev algebra $(\overline{\M}, \br_{\overline{\M}}, \prs_{\overline{\M}})$ by means of $(\xi, D, b_0, 0)$.
\end{Theorem}
\begin{proof}
 Let \( (\M, \br, \prs) \) be a flat Lorentzian  Malcev algebra, and assume that \( Z(\M) \) is degenerate. In this case, we define \( J = Z(\M) \cap Z(\M)^\bot = \K e \), where \( e \) is an isotropic vector. From Lemma \ref{Leprinc}, we deduce that \( \Ll_e = \Rr_e = 0 \), implying that \( J \) is a two-sided ideal with respect to the Levi-Civita product. Furthermore, the orthogonal complement \( J^\bot \) is also a two-sided ideal. According to Theorem \ref{dbex},  \( (\M, \br, \prs) \) is a double extension of a flat Euclidean Malcev algebra \( (\overline{\M}, \br_{\overline{\M}}, \prs_{\overline{\M}}) \) by means of \( (\xi, D,  b_0, \la) \). The bracket on \( \M \) is given by:
\[
[d, e] = \la e, \quad [d, u] = D(u)+ \langle b_0, u \rangle_{\overline{\M}} e, \quad [u,v] =   [u,v]_{\overline{\M}}  -\langle(\xi - \xi^*)(u), v \rangle_{\overline{\M}} e.
\]
Since \( e \in Z(\M) \), we conclude that \( \la = 0 \). 
\end{proof}
\ssbegin{Corollary}\label{SL}
  Let $(\M, \br, \prs)$    be a flat  Lorentzian solvable Malcev algebra such that $Z(\M)$ is degenerate. Then $(\M, \br, \prs)$  is a flat double extension of a flat Euclidean Lie algebra $(\overline{\M}, \br_{\overline{\M}}, \prs_{\overline{\M}})$ by means of $(\xi, D, b_0, 0)$ where $(\overline{\M}, \br_{\overline{\M}}, \prs_{\overline{\M}})$ is flat as a Lie algebra.  
\end{Corollary}
\begin{proof}
 According to Theorem \ref{Thpr},   $(\M, \br, \prs)$    is a flat double extension of a flat Euclidean Malcev algebra $(\overline{\M}, \br_{\overline{\M}}, \prs_{\overline{\M}})$ by means of $(\xi, D, b_0, 0)$. Since $(\M, \br)$ is a solvable Malcev algebra, it follows from the bracket relation \eqref{crochet} that the algebra $(\overline{\M}, \br_{\overline{\M}})$ is also a solvable Malcev algebra. According to Corollary \ref{co1}, it follows that, 
  $(\overline{\M}, \br_{\overline{\M}}, \prs_{\overline{\M}})$ is a   Euclidean Lie algebra and its flat as Lie algebra. 
\end{proof}

 \ssbegin{Theorem}\label{Thnilpotent}
  Let $(\M, \br, \prs)$ be a flat  Lorentzian nilpotent Malcev algebra. Then $(\M, \br, \prs)$    is a flat double extension of a flat Euclidean abelian Lie algebra $(\overline{\M}, \br_{\overline{\M}}, \prs_{\overline{\M}})$ by means of $\xi=-D$, $\xi^2=0$, $b_0\in \overline{\M}$.  Furthermore, $(\M,\br)$ is at most $3$-step nilpotent.
\end{Theorem}  

\begin{proof}
Since $(\M,\br)$ is  nilpotent, we have \( Z(\M) \cap [\M, \M] \neq \{0\} \). According to Proposition \ref{nilpotent}, for any \( a \in Z(\M) \), we have \( \Ll_a^3 = 0 \). 
If \( \Ll_a^2 = 0 \), then by Lemma \ref{nul}, it follows that \( \Ll_a =\Rr_a= 0 \) for all \( a \in Z(\M) \). From the relation \eqref{Ortho} we obtain that \( Z(\M) \subset [\M, \M]^\bot \), which implies that $$Z(\M) \cap [\M, \M] \subseteq Z(\M) \cap Z(\M)^\bot. $$ We put \( J := Z(\M) \cap Z(\M)^\bot = \K e \), we deduce that \( J \) is a totally isotropic subspace, and that we have  \( \Ll_e = \Rr_e = 0 \).
Now, we suppose that \( \Ll_a^2 \neq 0 \). Define 
\(
N = \bigcap_{b \in Z(\M)} \ker (\Ll_b^2).
\)
Since \( \Ll_b \) is skew-symmetric, we have 
\(
N^\perp = \sum_{b \in Z(\M)} \mathrm{Im} (\Ll_b^2).
\) 
Since  \( \Ll_a^3 = 0 \) for any \( a \in Z(\M) \), it follows that \( N^\bot \subset N \).
On the other hand, using equation \eqref{K}, we take \( v = w = a \) and \( u = z \in \M \), we obtain
\[
\Ll_a \circ \Ll_z \circ \Ll_a = \Ll_a \circ \Ll_a \circ \Ll_z.
\]
From relation \eqref{re1}, setting \( u = v = a \) and \( w = z \in \M \), we get \( \Ll_a \circ \Ll_z \circ \Ll_a = 0 \). Since \( \Ll_u \) is skew-symmetric for any \( u \in \M \), it follows that \( \Ll_z \circ \Ll_a \circ \Ll_a = 0 \). Consequently, \( \Rr_{a \bullet (a \bullet u)} = 0 \) for every \( u \in \M \), which implies that \( N^\bot \subset (\M \bullet \M)^\bot\) and hence \( \M \bullet \M \subset N \).

In addition to that,  we have  \( N^\bot \subset N \),  which means that we can write  \( \M \) as in the following decomposition 
$
\M := \K e \oplus U \oplus \K f,
$
where \(N^\bot=\K e,\; N = \K e \oplus U \), and \( f \) is a totally isotropic vector satisfying \( \langle e, f \rangle = 1 \).  We deduce that 
\[
e \bullet u = g(u)e + F(u), \quad e \bullet f = \lambda e + c_0, \quad z \bullet e = 0
\]
for all \( u \in U \) and \( z \in \M \), where \( \lambda \in \K \), \( g \in U^* \), \( F \in \mathrm{End}(U) \), and \( c_0 \in U \). Since \( \Ll_z \) is skew-symmetric for any \( z \in \M \), we obtain
\begin{align*}
&\langle e \bullet f, f \rangle = \lambda = 0, \quad \langle e \bullet f, u \rangle = \langle c_0, u \rangle = -\langle f, e \bullet u \rangle = -g(u),\\
&\langle e\bullet u,v\rangle=\langle F(u),v\rangle=-\langle u, e\bullet v\rangle=-\langle u, F(v)\rangle.
\end{align*}
Thus, we conclude that we have 
\[
e \bullet u = -\langle c_0, u \rangle e + F(u), \quad e \bullet f = c_0, \quad z \bullet e = 0,
\]
where \( F \) is skew-symmetric with respect to \( \prs \). Furthermore, for any \( u \in U \)  we have
\[
\ad^n_e(u) = -\langle c_0, F^{n-1}(u) \rangle e + F^n(u).
\]
Since \( (\M, \br) \) is nilpotent, it follows that \( F \) is nilpotent. However \( U \) is Euclidean and \( F \) is both skew-symmetric and nilpotent, we deduce that \( F = 0 \).

Now using relation \eqref{K}, computing \( K(f, e, e) u \), for any \( u \in U \), we obtain
\[
K(f, e, e) u = \langle c_0, c_0 \rangle \langle c_0, u \rangle.
\]
Since \( (\M, \br, \prs) \) is flat and \( U \) is Euclidean, it follows that \( c_0 = 0 \). Thus we conclude that we have  \( e \bullet z = z \bullet e = 0 \), for all $z\in \M$.
In both cases, there exists an isotropic vector $e$
 with \( \Ll_e = \Rr_e = 0 \). Defining \( J := Z(\M) \cap Z(\M)^\bot = \K e \). Thus, it is clear that  \( J \) is a two-sided ideal with respect to the Levi-Civita product. Moreover, its orthogonal \( J^\bot \) is also a two-sided ideal. According to Theorem \ref{dbex}, \( \M \) admits the decomposition
$
\M = \K e \oplus \overline{\M} \oplus \K d,
$
which   \( \M \) is a flat  double extension of the flat Euclidean Malcev algebra \( (\overline{\M}, \br_{\overline{\M}}, \prs_{\overline{\M}}) \) by means of \( (\xi, D, b_0, \la) \). The Malcev bracket on \( \M \) is given by
\begin{equation}
[d, u] =   D(u)+\langle b_0, u \rangle_{\overline{\M}} e, \quad [u,v] = [u,v]_{\overline{\M}}-\langle(\xi - \xi^*)(u), v \rangle_{\overline{\M}}  e ,
\label{Malbra}\end{equation}
for all \( u,v \in \overline{\M} \). Since \( (\M, \br) \) is nilpotent, it is clear that both 
 $D$ and \( (\overline{\M}, \br_{\overline{\M}}) \) are nilpotent. It follows that \( (\overline{\M}, \br, \prs) \) is a flat Euclidean nilpotent Malcev algebra. By Corollary \ref{Nul} we deduce that \( (\overline{\M}, \br) \) is an abelian Lie algebra. Consequently, the relations in system \eqref{conflat1} reduce to the following system:
\begin{equation}
\begin{cases}
\xi^{*}\circ \xi\circ D 
= \xi^{*}\circ (D+\xi)\circ \xi 
= - D^{*}\circ \xi^{*}\circ \xi 
, \\[0.3em]
\xi\circ D^{2} 
= (D+\xi)^{2}\circ \xi 
= (D+\xi)\circ \xi\circ D, \\[0.3em]
\xi^{*}((D+\xi)(b_0)) 
= D^{*}(\xi(b_0)) 
= 0.
\label{sy1}\end{cases}
\end{equation}

Put \( A = D + \xi \). We have  
\begin{align*}
    A^3 &= (D+\xi)^2\circ(D+\xi) \\ 
        &= (D+\xi)^2D + (D+\xi)^2\circ \xi \\ 
        &= D^3 + (D+\xi) \circ \xi \circ D + \xi \circ D^2 + (D+\xi)^2\circ \xi \\ 
        &= D^3 + 3\xi \circ D^2,
\end{align*}  
where we used the relation \( (D+\xi) \circ \xi \circ D = (D+\xi)^2 \circ \xi = \xi \circ D^2 \).  
By induction, we deduce that for any \( n \in \mathbb{N}^* \), we have
\[
A^{2n+1} = D^{2n+1} + (2n+1) \xi \circ D^{2n}.
\]  

Indeed for \( n = 1 \), the formula holds as shown above. Now assuming that it holds for some \( n \), we prove it for \( n+1 \)  
\begin{align*}
    A^{2n+3} &= (D+\xi)^2 \circ A^{2n+1} \\  
    &= (D+\xi)^2 \circ \big( D^{2n+1} + (2n+1)\xi \circ D^{2n} \big) \\  
    &= (D+\xi)^2 \circ D^{2n+1} + (2n+1)(D+\xi)^2 \circ \xi \circ D^{2n} \\  
    &= D^{2n+3} + (D+\xi) \circ \xi \circ D \circ D^{2n} + \xi \circ D^{2n+2} + (2n+1)(D+\xi)^2 \circ \xi \circ D^{2n} \\  
    &= D^{2n+3} + (2n+3) \xi \circ D^{2n+2}.
\end{align*}  

because \( (D+\xi) \circ \xi \circ D = (D+\xi)^2 \circ \xi = \xi \circ D^2 \). Since \( D \) is nilpotent, it follows that \( A \) is also nilpotent. Moreover, since \( A \) is skew-symmetric with respect to \( \prs_{\overline{\M}} \), we conclude that \( A = 0 \), which implies \( D = -\xi \).  
From  \eqref{sy1}, we have \( (\xi^2)^* \xi = 0 \), which leads to \( \xi^2 = D^2 = 0 \). Consequently, using  the Malcev bracket \eqref{Malbra}, we deduce that \( (\M, \br) \) is at most nilpotent of order 3.  
This completes the proof.  
\end{proof}

\ssbegin{Proposition}\label{PLie}
 Let $(\M, \br, \prs)$ be a flat Lorentzian  nilpotent Malcev algebra, then $(\M, \br, \prs)$ is a flat Lorentzian nilpotent Lie algebra.  
\end{Proposition}

\begin{proof}  
According to Theorem \ref{Thnilpotent}, $(\M, \br, \prs)$ is a double extension of the Euclidean abelian Lie algebra $(\overline{\M}, \br_{\overline{\M}}, \prs_{\overline{\M}})$ by means of $\xi=-D$, $\xi^2=0$, $b_0\in \overline{\M}$. By Koszul’s formula \eqref{lv}, the Levi-Civita product is given by:  
\begin{equation}
\begin{cases} 
    e \bullet u = u \bullet e = e \bullet e = e \bullet d = d \bullet e = 0, \\
    u \bullet v = -\langle \xi(u), v \rangle_{\overline{\M}} e, \\
    d \bullet d =- b_0, \\
    d \bullet u = \langle b_0, u \rangle_{\overline{\M}} e, \\
    u \bullet d =  \xi(u),
\end{cases} 
\end{equation}
for all \( u,v \in \overline{\M} \).  
We now prove that \( (\M, \br, \prs) \) is a flat Lorentzian Lie algebra, i.e., its curvature tensor vanishes identically,  
\[
K_0(x,y)z = \Ll_{[x,y]}(z) - \Ll_x \circ\Ll_y (z) + \Ll_y\circ \Ll_x (z), \quad \forall x,y,z \in \M.
\]  
We verify \( K(x,y)z = 0 \) for all cases:\\  
1. Case \( K(d, u, d) = 0 \) for any \( u \in \overline{\M} \): 
\begin{align*}
    K(d, u, d) &= \Ll_{[d, u]}(d) - \Ll_d \circ\Ll_u (d)+\Ll_u\circ \Ll_d (d) \\  
    &= [d, u] \bullet d - d \bullet (u \bullet d) + u \bullet (d \bullet d) \\  
    &= -\xi(u) \bullet d - d \bullet \xi(u) - u \bullet b_0 \\  
    &= -\xi^2(u) -\langle b_0, \xi(u) \rangle_{\overline{\M}} e + \langle \xi(u), b_0 \rangle_{\overline{\M}} e \\  
    &= 0.  
\end{align*}  

2. Case \( K(d, u, v) = 0 \) for any \( u, v \in \overline{\M} \):  
\begin{align*}
    K(d, u, v) &= \Ll_{[d, u]}(v) - \Ll_d \circ\Ll_u (v) +\Ll_u \circ\Ll_d (v) \\  
    &= [d, u] \bullet v - d \bullet (u \bullet v) + u \bullet (d \bullet v) \\  
    &= -\xi(u) \bullet v \\  
    &= \langle \xi^2 (u), v \rangle_{\overline{\M}} e \\  
    &= 0.  
\end{align*}  

3. Case \( K(u, v, d) = 0 \) for any \( u, v \in \overline{\M} \):  
\begin{align*}
    K(u, v, d) &= \Ll_{[u, v]}(d) - \Ll_u \circ\Ll_v (d) + \Ll_v \circ\Ll_u (d) \\  
    &= [u, v] \bullet d - u \bullet (v \bullet d) + v \bullet (u \bullet d) \\  
    &=- u \bullet \xi(v) + v \bullet \xi(u) \\  
    &= \langle \xi(u), \xi(v) \rangle_{\overline{\M}} e - \langle \xi(v), \xi(u) \rangle_{\overline{\M}} e \\  
    &= 0.  
\end{align*}  

4. Case \( K(u, v, w) = 0 \) for any \( u, v, w \in \overline{\M} \):  
\begin{align*}
    K(u, v, w) &= \Ll_{[u, v]}(w) - \Ll_u \circ \Ll_v (w) + \Ll_v \circ \Ll_u (w) \\   
    &= [u, v] \bullet w - u \bullet (v \bullet w) + v \bullet (u \bullet w) \\  
    &= 0.  
\end{align*}  
Since the curvature tensor vanishes in all cases, we conclude that  $(\M,\br,\prs)$ is a flat Lorentzian Lie algebra.
\end{proof}
\ssbegin{Remark}
    Note that  every flat  Lorentzian nilpotent Malcev algebra \((\M, \br, \prs)\) is in fact a flat Lorentzian nilpotent Lie algebra.  In \cite{Bajo}, the authors carried out a classification of all flat  Lorentzian nilpotent Lie algebras.
\end{Remark}

\bibliographystyle{elsarticle-num}

\end{document}